\documentclass{amsart}

\usepackage[utf8]{inputenc}
\usepackage{amsmath, amsthm, amssymb, mathtools, empheq, cite}
\usepackage{tikz, color}
\usepackage[colorlinks=true, linkcolor=black, citecolor=black, filecolor=black, urlcolor=black, pdfborder={0 0 0}, linktoc=all]{hyperref}

\usepackage{todonotes}
\usepackage{cleveref}
\usepackage{caption}
\usepackage{subcaption}
\usetikzlibrary{calc, math}
\usepackage{diagbox}
\usetikzlibrary{intersections}

\makeatletter

\newtheorem{theorem}{Theorem}[section]
\newtheorem{corollary}[theorem]{Corollary}
\newtheorem{proposition}[theorem]{Proposition}
\newtheorem{lemma}[theorem]{Lemma}

\theoremstyle{definition}
\newtheorem{definition}[theorem]{Definition}
\newtheorem{example}[theorem]{Example}
\newtheorem{question}[theorem]{Question}
\newtheorem{remark}[theorem]{Remark}

\newtheorem{observation}[theorem]{Observation}
\newtheorem{convention}[theorem]{Convention}

\definecolor{blue}{rgb}{0, 0.445, 0.695}
\definecolor{bluishgreen}{rgb}{0, 0.626, 0.456}
\definecolor{red}{rgb}{0.896, 0.395, 0}
\definecolor{purple}{rgb}{0.783, 0.464, 0.640}
\definecolor{skyblue}{rgb}{0.359, 0.752, 0.973}
\definecolor{orange}{rgb}{0.999, 0.706, 0.0}
\definecolor{yellow}{rgb}{0.937, 0.890, 0.258}
\definecolor{aa}{RGB}{0,191,255}

\title{Geometric realizations of $\nu$-associahedra via brick polyhedra}
\author[C.~Ceballos]{Cesar Ceballos}
\address[C.~Ceballos]{Institute of Geometry, University of Technology Graz, Austria}
\email{cesar.ceballos@tugraz.at}

\author[M.~M\"uller]{Matthias M\"uller}
\address[M.~M\"uller]{Institute of Geometry, University of Technology Graz, Austria}
\email{matthias.mueller@tugraz.at}

\thanks{Cesar Ceballos was supported by the Austrian Science Fund FWF, grants P 33278 and I 5788.}

\usepackage[x11names, table]{xcolor}
\usepackage{graphicx}
\usepackage{pdfpages}

\usepackage{pgf}
\usepackage{pgfplots}
\pgfplotsset{compat=1.8}
\usepackage{mathtools}
\usepackage{tikz}
\usetikzlibrary{arrows}	
\usetikzlibrary{patterns}

\usepackage[framemethod=TikZ]{mdframed}		
\usetikzlibrary{calc}
\usepackage{pgfplots}
\usepgfplotslibrary{groupplots}
\usepackage{tikz-cd}

\usepackage{amsmath,amssymb}
\usepackage{mathpazo}
\usepackage{ifthen}
\usepackage{tikz,pgf,pgfmath}
	\usetikzlibrary{shapes,calc,decorations.pathmorphing}
	\pgfdeclarelayer{background}
	\pgfdeclarelayer{middle}
	\pgfsetlayers{background,middle,main}
	\tikzstyle{fnode}=[fill=black,draw=black,circle,scale=\s]
	\tikzstyle{back}=[thin,dashed]
	\tikzstyle{edge}=[line width=.75pt]
	\newcommand{\red}{red!50!white}
	\newcommand{\blue}{blue!50!white}

	\newcommand{\green}{green!50!gray}
	\newcommand{\gray}{white!50!gray}
	\tikzstyle{facet}=[fill=\green,fill opacity=0.6]
	\tikzstyle{facetR}=[fill=\red,fill opacity=0.8]
	\tikzstyle{facetL}=[fill=\blue,fill opacity=0.8]
	\tikzstyle{pathnode}=[inner sep=.9pt]

\newcounter{i}
\newcounter{n}
\newcounter{m}
\newcounter{curx}
\newcounter{cury}
\newcounter{orient}
\newcommand{\nuTree}[7]{
	\begin{tikzpicture}
		\def\dx{#5};
		\def\s{3*\dx};
		\setcounter{n}{0}
		\setcounter{m}{0}
		\setcounter{orient}{0}
		\setcounter{curx}{0}
		\setcounter{cury}{0}
		\foreach \step in {#1}{
			\ifthenelse{\equal{\theorient}{0}}{
				\addtocounter{n}{\step}
				\setcounter{orient}{1}
			}{
				\addtocounter{m}{\step}
				\setcounter{orient}{0}
			}
		}
		\setcounter{i}{0}
		\setcounter{orient}{0}
		\foreach \step in {#1}{
			\ifthenelse{\equal{\theorient}{0}}{
				\coordinate(q\thei) at (\thecurx*\dx,\thecury*\dx);
				\stepcounter{i}
				\coordinate(q\thei) at (\thecurx*\dx,{(\thecury+\step)*\dx});
				\stepcounter{i}
				\draw[\gray,dashed](\thecurx*\dx,\thecury*\dx) -- (\thecurx*\dx,{(\thecury+\step)*\dx});
				\begin{pgfonlayer}{background}
					\foreach \k in {1,...,\step}{
						\draw[\gray,dashed](0*\dx,{(\thecury+\k)*\dx}) -- (\thecurx*\dx,{(\thecury+\k)*\dx});
					}
				\end{pgfonlayer}
				\addtocounter{cury}{\step}
				\setcounter{orient}{1}
			}{
				\draw[\gray,dashed](\thecurx*\dx,\thecury*\dx) -- ({(\thecurx+\step)*\dx},\thecury*\dx);
				\begin{pgfonlayer}{background}
					\foreach \k in {0,...,\step}{
						\ifthenelse{\equal{\k}{\step}}{
						}{
							\draw[\gray,dashed]({(\thecurx+\k)*\dx},\thecury*\dx) -- ({(\thecurx+\k)*\dx},\then*\dx);
						}
					}
				\end{pgfonlayer}
				\addtocounter{curx}{\step}
				\setcounter{orient}{0}
			}
		}
		\ifthenelse{\equal{#7}{}}{}{
			\coordinate(zero) at (0*\dx,0*\dx);
			\coordinate(q\thei) at(0*\dx,\then*\dx);
			\addtocounter{i}{-1}
			\foreach \k in {0,2,...,\thei}{
				\pgfmathparse{\k+1}
				\let\res\pgfmathresult
				\begin{pgfonlayer}{middle}
					\fill[#7] (zero |- q\k) -- (q\k) -- (q\res) -- (zero |- q\res) -- cycle;
				\end{pgfonlayer}
			}
		}
		\setcounter{i}{1}
		\foreach \a/\b/\c/\d in {#2}{
			\draw(\a*\dx,\b*\dx) node[fnode,draw=\c,fill=\c](p\thei){};
			\ifthenelse{\equal{\d}{d}}{
				\draw[orange!50!gray,decorate,decoration={snake,amplitude=.5,segment length=2.5}](p\thei) -- (\a*\dx,{(\then+.5)*\dx});
			}{
			}
			\stepcounter{i}
		}
		\foreach \a/\b in {#3}{
			\draw[line width=3*\dx pt](\a) -- (\b);
		}
 		\foreach \a/\b/\c in {#4}{
			\draw[\c,line width=.5pt](\a*\dx,\b*\dx) circle(.25*\dx);
 		}
		\foreach \a/\b/\c/\d in {#6}{
			\draw[\d](\a*\dx,\b*\dx) node[scale=.75,anchor=west]{\c};
		}
	\end{tikzpicture}
}

\newcommand{\ccross}[1][black]{\raisebox{-.15cm}{\includegraphics[scale=.9]{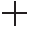}}}
\newcommand{\celbow}[1][black]{\raisebox{-.15cm}{\includegraphics[scale=.9]{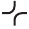}}}
\definecolor{red}{rgb}{0.9, 0.1, 0.1}


\begin{document}


\begin{abstract}
Brick polytopes constitute a remarkable family of polytopes associated to the spherical subword complexes of Knutson and Miller. They were introduced for finite Coxeter groups by Pilaud and Stump, who used them to produce geometric realizations of generalized associahedra arising from the theory of cluster algebras of finite types. In this paper, we present an application of the vast generalization of brick polyhedra for general subword complexes (not necessarily spherical) recently introduced by Jahn and Stump. 

More precisely, we show that the $\nu$-associahedron, a polytopal complex whose edge graph is the Hasse diagram of the $\nu$-Tamari lattice introduced by Préville-Ratelle and Viennot, can be geometrically realized as the complex of bounded faces of the brick polyhedron of a well chosen subword complex. We also present a suitable projection to the appropriate dimension, which leads to an elegant vertex-coordinate description.
\end{abstract}



\maketitle

\tableofcontents

\section{Introduction}
The purpose of this work is to present an application of brick polyhedra of general subword complexes to produce geometric realizations of $\nu$-associahedra.

There are several known connections between brick polytopes and generalizations of the associahedron. 
A main core for such connections is Knutson and Miller's theory of subword complexes. 
Subword complexes are certain simplicial complexes motivated by the study of Gr\"obner geometry of Schubert varieties~\cite{KnutsonMiller_Groebner_2005, knutson_miller}.  
One of the first connections between subword complexes and associahedra was discovered by Pilaud and Pocchiola in~\cite{pilaud_multitriangulations_2012} using a slightly different terminology (of sorting networks), which was rediscovered using the subword complex terminology in~\cite{stump_newperspective_2011}.
A generalization for arbitrary finite Coxeter groups is due to Ceballos, Labb\'e and Stump in~\cite{ceballos_and_labbe}, who showed that $c$-cluster complexes arising in the theory of cluster algebras of finite type~\cite{clusteralgebrasII} can be obtained as well chosen subword complexes. 
The dual graph of the cluster complex, also known as the mutation graph for cluster algebras, is the edge graph of a well known polytope called the generalized associahedron~\cite{fomin_ysystems_2003,chapoton_polytopal_2002}. 

This last connection motivated the introduction of brick polytopes for spherical subword complexes by Pilaud and Stump~\cite{bpofssc}, who generalized the notion of brick polytopes in type $A$ by Pilaud and Santos in~\cite{bpofsn}. One of the main results in~\cite{bpofssc} provides a geometric realization of the generalized associahedron as the brick polytope of a spherical subword complex.
Later on, Jahn and Stump presented a generalization of brick polyhedra for arbitrary subword complexes (not necessarily spherical) of finite type~\cite{bpolyhedra}, who nicely connected them to the combinatorics and geometry of Bruhat intervals and Bruhat cones in Coxeter groups.  
Our work presents the first application of brick polyhedra to produce geometric realizations of $\nu$-associahedra.  
    \begin{figure}[!h]
    \centering
    \includegraphics[width=0.7\textwidth]{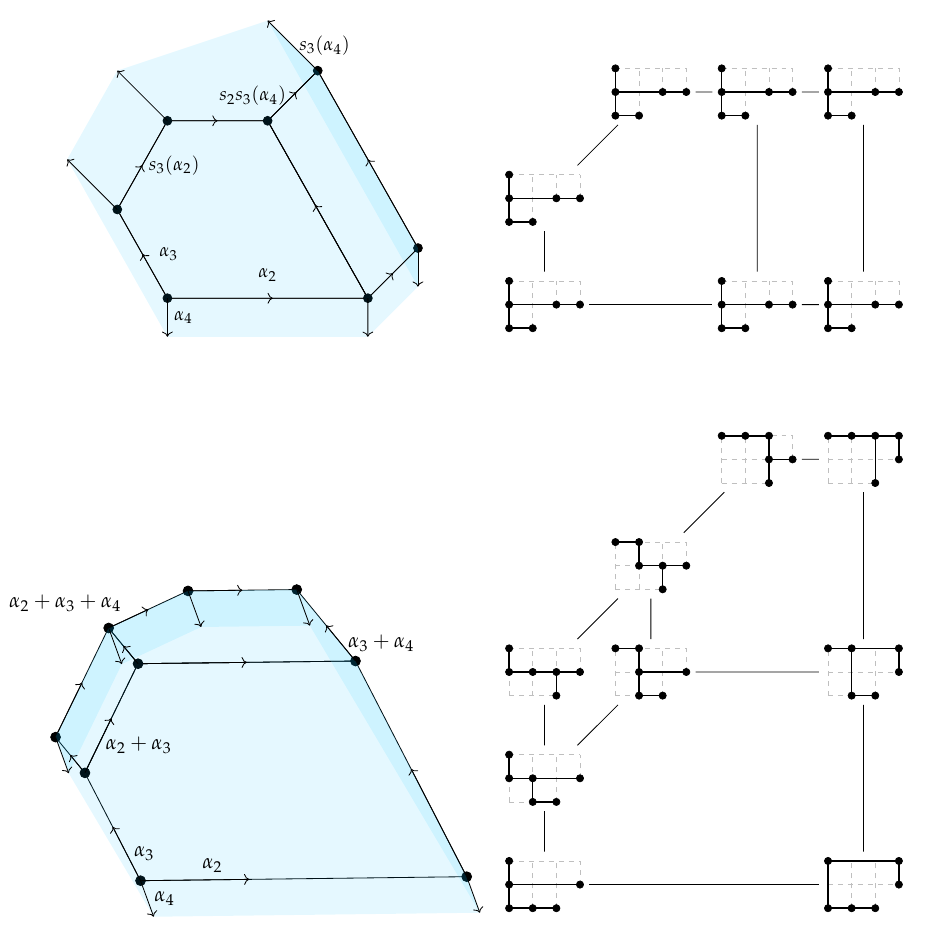} 
    \caption{Comparison of the $\nu$-brick polyhedron and  $\nu$-associahedron for $\nu=ENEEN$ (top) and for $\nu=EENEN$ (bottom).}
    \label{ENEENcomp_intro}
    \end{figure}
Given a lattice path $\nu$, consisting of finitely many north steps N and east steps~E, the \mbox{$\nu$-associahedron}~\cite{CPS19} is a polytopal complex whose edge graph is the Hasse diagram of the $\nu$-Tamari lattice introduced by Préville-Ratelle and Viennot in~\cite{preville_vTamari_2017}, and whose face poset is the poset of interior faces of the corresponding $\nu$-Tamari complex~\cite{CPS19} ordered by reverse inclusion. The case $\nu=(NE)^n$ recovers the classical associahedron, whose edge graph is the Hasse diagram of the classical Tamari lattice, which can be defined as the rotation poset of rooted binary trees and plays a fundamental role in many areas in mathematics, computer science and physics. 
The case $\nu=(NE^m)^n$ recovers the $m$-Tamari lattices of Bergeron~\cite{bergeron_higher_2012}, whose interval enumeration has beautiful conjectural connections to the theory of trivariate diagonal harmonics in representation theory. 
Using techniques from tropical geometry, Ceballos, Sarmiento and Padrol produced the first geometric realizations of $\nu$-associahedra~\cite{CPS19}, solving an open problem of Bergeron in this more general set up. 

In this paper, we present a second geometric realization of the $\nu$-associahedron as the complex of bounded faces of the brick polyhedron of a well chosen (non-spherical) subword complex. We also provide a suitable projection, in the special case where $\nu$ has non consecutive north steps, which provides a realization of the appropriate dimension, with a beautiful and elegant vertex-coordinate description. The brick polyhedron and the projection are illustrated for two examples in~\Cref{ENEENcomp_intro}. Some $3$-dimensional examples (resulting dimension after projecting) are illustrated in Figures~\ref{fig_ENENEENproj} and \ref{fig_nu_associahedra}. 

Brick polyhedra and $\nu$-associahedra are central structures in their respective fields. The contributions in this paper do not only extend the notorious connection between brick polytopes related to type $A$ cluster algebras and classical associahedra, but provides new insights for advances in both areas:
\begin{enumerate}
    \item \emph{Brick polyhedra}: 
    We provide the first known explicit combinatorial family of examples of brick polyhedra. The combinatorial understanding of $\nu$-Tamari lattices and $\nu$-associahedra is a powerful toolbox that provides intuition for a deeper understanding of the geometric structure of general brick polyhedra (see~\Cref{sec_faces_brick_polyhedra} and~\Cref{facesall}),
    motivating further questions for research in this area (see for instance~\Cref{question_faces_brick_polyhedra}).
    \item \emph{$\nu$-associahedra and generalizations}: 
    Our description of $\nu$-associahedra in terms of brick polyhedra opens new avenues of research in a much wider context, making this paper the starting point for various research directions. 
    A very promising direction of research is related to the family of framing lattices arising from triangulations of flow polytopes~\cite{bell_framing_2024}. This vast generalization captures many important lattices such as $\nu$-Tamari lattices, type~$A$ Cambrian lattices, Grassmann- and grid Tamari lattices, the $s$-weak order and  more, under the same umbrella.  
    The connections in this paper motivate the introduction of framing polyhedra~\cite{bell_framingtopes}, which play the role of brick polyhedra for framing lattices, and relate Bruhat cones to normal cones of flow polytopes. 
    We believe that framing polyhedra will be a fundamental tool in the study of framing lattices.
\end{enumerate}

\section{Brick polyhedra}
Throughout this work, we restrict our study to finite Coxeter groups, subword complexes and brick polyhedra of type $A$.

A Coxeter system \((W, S)\) of type \( A_n \) consists of the Coxeter group \( W := \mathcal{S}_{n+1} \) of permutations of $[n+1]$, which acts on the space \( \{ x \in \mathbb{R}^{n+1} \mid x_1 + \dots + x_{n+1} = 0 \} \) by permuting coordinates. 
It is finitely generated by simple transpositions \( S := \{ s_p \mid p \in [n] \} \) with \( s_p = (p, p + 1) \). The root system is defined by \( \Phi = \{ e_p - e_q \mid p \neq q \in [n+1] \} \), and can be partitioned into positive roots \( \Phi^+ = \{ e_i - e_j \mid 1 \leq i < j \leq n+1 \} \) and negative roots \( \Phi^- = \{ e_j - e_i \mid 1 \leq i < j \leq n+1 \} \). 
The simple roots are \( \Delta = \{ \alpha_p:=e_p - e_{p+1} \mid p \in [n] \} \) and the fundamental weights are \( \nabla = \{ \omega_p:=\sum_{q \leq p} e_q \mid p \in [n] \} \).

\begin{definition}[Subword complex \cite{knutson_miller}]
For a Coxeter system $(W,S)$, let $Q=(q_1,...,q_m)$ be a word in the generators $S$ of $W$ and let $w \in W$. The \emph{subword complex}~$\mathcal{SC}(Q,w)$ is the simplicial complex whose facets are subsets $I\subseteq [m]$ such that $Q_{[m]\setminus I}$ is a reduced expression for $w$. Here $Q_J$ denotes the subword of $Q$ with positions at $J$.
\end{definition}

We can now define two important functions associated with brick polyhedra.

\begin{definition}[Root and Weight Function \cite{ceballos_and_labbe,bpofssc}]
     Given a facet $I$ of~$\mathcal{SC}(Q,w)$, the \emph{root function} is the map $r$$(I,\cdot): [m] \xrightarrow{} \Phi$ defined by $r$$(I,k):=\prod Q_{ \{1,...,k-1\}\setminus I} (\alpha_{q_k})$. 
     We call ${R}(I) := \{\{ $$r$$(I,i)\mid i \in I\}\}$ 
     the \emph{root configuration} of $I$. The other function is the \emph{weight function} $\omega(I,\cdot): [m] \xrightarrow{} \Phi$, defined by $\omega(I,k):=\prod Q_{ \{1,...,k-1\}\setminus I} (\omega_{q_k})$.
\end{definition}

\begin{definition}[Bruhat~cone \cite{bpolyhedra}]\label{bpdrei_cone}
The \emph{Bruhat~cone} of a non-empty subword complex $\mathcal{SC}(Q,w)$ is defined by 
\[
\mathcal{C}^+(w,\operatorname{Dem}(Q)):=\operatorname{cone}\{ \beta \in \Phi^+ \mid w \prec_B s_\beta w \leq_B \operatorname{Dem}(Q)\},
\]
where $\operatorname{Dem}(Q)=\text{max}_{\leq_B} \{\prod Q_X \mid X \subseteq \{1,...,m\}\}$ denotes the Demazure product of~$Q$~\cite[Lemma 3.4 (1)]{knutson_miller}, and $\leq_B,\prec_B$ denote the Bruhat order and its cover relation.
\end{definition}

\begin{proposition}[{\cite{bpolyhedra}}]
The Bruhat cone can be computed as
\begin{align*}
\mathcal{C}^+(w,\operatorname{Dem}(Q)) =
\bigcap_{\text{Facet } J \in \mathcal{SC}(Q,w)} \operatorname{cone} ~R(J).
\end{align*}
\end{proposition}

\begin{definition}[Brick polyhedron \cite{bpolyhedra}]\label{bpdrei}
The \emph{brick vector} of a facet $I \in \mathcal{SC}(Q,w)$~is 
\[
b(I):=-\sum_{k=1}^{m}\omega(I,k).
\]
The \emph{brick polyhedron}~$\mathcal{B}(Q,w)$ is the Minkowski sum of the convex hull of all brick vectors and the Bruhat~cone:
\begin{center}
$\mathcal{B}(Q,w):= \text{conv}\{b(I)\mid$ I facet of $\mathcal{SC}(Q,w)\}+\mathcal{C}^+(w,\operatorname{Dem}(Q))$.
\end{center}
\end{definition}

At first glance, brick polyhedra do not seem natural, but they turn out to have very nice properties related to the combinatorics and geometry of the corresponding subword complex~\cite{bpolyhedra}. 
We aim to relate this to the combinatorics and geometry of $\nu$-associahedra.

\section{The $\nu$-Tamari lattice and the $\nu$-associahedron}

We start by introducing the concept of $\nu$-Tamari lattices using the conventions in~\cite{ceballos_vTamari_subword_2020}.
We denote by $\nu$ a lattice path with finitely many east and north steps. Let $F_\nu$ be the Ferrers diagram weakly above $\nu$, inside the smallest rectangle containing $\nu$. We denote by $A_\nu$ the set of lattice points weakly above $\nu$, which are inside $F_\nu$. 
For a lattice point $p\in A_\nu$, we denote by $d(p)$ the lattice distance from $p$ to the top-left corner of $F_\nu$. 

\begin{definition}[$\nu$-tree~\cite{ceballos_vTamari_subword_2020}]
    For $ p, q \in A_\nu $, we say that $p$ and $q$ are $\nu$-incompatible, denoted $p \not\sim q$, if and only if $ p $ is southwest (SW) of $ q $ or $ p $ is northeast (NE) of $ q $, and the smallest rectangle containing $ p $ and $q $ lies completely inside $ F_\nu $. A \emph{$\nu$-tree} is a maximal collection of pairwise $\nu$-compatible elements in $A_\nu$. 
    Its elements are called \emph{nodes} and the top left corner is called \emph{root}.
    We associate a rooted binary tree to each $\nu$-tree~$T$ by connecting every $p\in T$ other than the root to the next in north or west direction, see Figure~\ref{nusc} (Left).
\end{definition}

\begin{definition}[$\nu$-Tamari lattice~\cite{ceballos_vTamari_subword_2020}]
 Two $\nu$-trees $T,T'$ are related by a \emph{right rotation} (or \emph{increasing flip}) if~$T'$ can be obtained from $T$ by exchanging $q\in T$ with~$q'\in T'$ as shown in Figure \ref{rotation} with $p,r\in T,T'$. The \emph{$\nu$-Tamari lattice} is the rotation poset of $\nu$-trees. An example of the Hasse diagram of the $\nu$-Tamari lattice for $\nu=ENEEN$ is the edge graph of Figure~\ref{ENEENasso}.
\end{definition}
    \begin{figure}[h]
    \centering
    \includegraphics[width=0.3\textwidth]{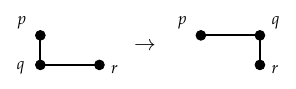} 
    \caption{Right rotation.}
    \label{rotation}
    \end{figure}
\begin{definition}[$\nu$-Tamari Complex \cite{ceballos_vTamari_subword_2020}]
    The \emph{$\nu$-Tamari complex} is the simplicial complex~$\mathcal{TC}(\nu$) of pairwise $\nu$-compatible sets in $A_\nu$. The dimension of a face $I$ is $\text{dim}(I)=|I|-1$. The facets are the $\nu$-trees.
\end{definition}
\begin{definition}[$\nu$-subword complex $\mathcal{SC}(Q_\nu,w_\nu)$ \cite{ceballos_vTamari_subword_2020}]\label{def3}
Given a lattice path $\nu$ we label each lattice point $p \in A_\nu$ by the transposition $s_{d(p)+1}$, see Figure~\ref{nusc} (Middle) for an example. Furthermore, define $Q_\nu$ as the word obtained by reading the associated transpositions from bottom to top, and the columns from left to right. The element~$w_\nu$ is the product of transpositions in the complement of a $\nu$-tree, see Figure~\ref{nusc} (Right). The complements of $\nu$-trees are reduced expressions of $w_\nu$ in $Q_\nu$ and the effect of a rotation keeps $w_\nu$ constant.
\end{definition}
\begin{figure}[h]
    \centering
    \includegraphics[width=0.8\textwidth]{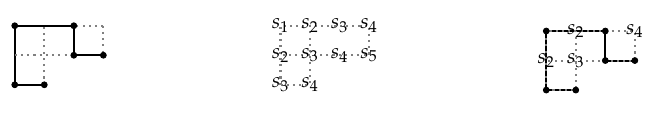} 
    \caption{Left: a $\nu$-tree for $\nu=ENEEN$. Middle: lattice points~$A_\nu$ labeled by transpositions; the corresponding word is $Q_\nu=(s_3,s_2,s_1,s_4,s_3,s_2,s_4,s_3,s_5,s_4)$. Right: complement of $\nu$-tree and its corresponding element $w_\nu=s_2s_3s_2s_4=[1,4,3,5,2,6]$.}
    \label{nusc}
\end{figure}

The reason why $w_\nu$ is independent of the choice of $T$ follows from the connection between $\nu$-trees and pipe dreams. 

\begin{definition}[Pipe dream \cite{lofap}, \cite{CPS19}]
A \emph{pipe dream} $P$ is a filling of a triangular shape with crosses~\ccross{} and elbows~\celbow{}, the lines are called \emph{pipes}. 
A pipe dream is called \emph{reduced} if every pair of pipes crosses at most once. 
We label the pipes entering on the left from top to bottom with the numbers from $1$ to $n$. The permutation $w(P)$ is the exiting permutation of pipes on the top of the figure, in one line notation. \Cref{pipe} shows an example of two reduced pipe dreams with exiting permutation $[1,4,3,5,2,6]$. 

\begin{figure}[h]
    \centering
    \includegraphics[width=0.6\textwidth]{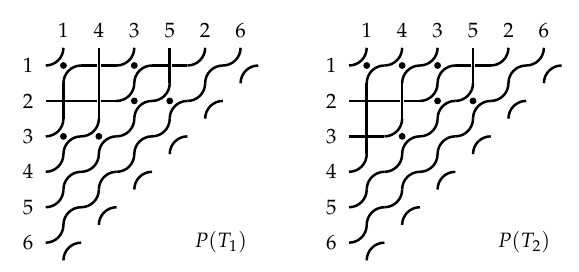}
    \caption{Pipe dreams for $\nu$-trees $T_1$ and $T_2$.}
    \label{pipe}
\end{figure}

To each $\nu$-tree $T$ we associate a pipe dream $P(T)$ by placing elbows at all nodes of the $\nu$-tree and outside~$F_\nu$. This is illustrated in Figure~\ref{pipe} for two $\nu$-trees related by a rotation.
Note that the exiting permutation remains unchanged, because the action of a tree rotation on pipe dreams exchanges one elbow between two pipes $i$ and $j$ with the unique crossing between pipes $i$ and $j$ (in our example, pipes 3 and~4).
\end{definition}

\begin{proposition}[{\cite{ceballos_vTamari_subword_2020}}]
The map sending $T\rightarrow P(T)$ is a bijection between $\nu$-trees and reduced pipe dreams with exiting permutation $w_\nu$.    
\end{proposition}

This connection allows us to provide a nice description of the $\nu$-Tamari complex as a well chosen subword complex. 

\begin{theorem}[{\cite{ceballos_vTamari_subword_2020}}]
    The $\nu$-subword complex $\mathcal{SC}(Q_\nu, w_\nu)$ is isomorphic to the $\nu$-Tamari complex. 
\end{theorem}
The interior faces of the $\nu$-Tamari complex can be characterized as follows.
\begin{definition}[\cite{ceballos_sweakorderII}]\label{ascent}
A node $q$ in a $\nu$-tree $T$ is called an ascent if there exists a node in $T$ to the north and another to the east of $q$. Equivalently, ascents of $T$ are the nodes of $T$ on which we can apply a right rotation.
\end{definition}

\begin{lemma}[\cite{ceballos_sweakorderII}]
The interior faces $I$ of the $\nu$-Tamari complex are in bijective correspondence with pairs $(T,A)$, where $T$ is a $\nu$-tree and $A$ is a subset of its ascents, via the map $I=T \setminus A$.
\end{lemma}

As we can observe from~\Cref{ENEENasso}, the $\nu$-Tamari lattice has a very rich underlying geometric structure.  
Its Hasse diagram can be geometrically realized as the edge graph of a polytopal complex called the $\nu$-associahedron \cite{CPS19}. The construction in \cite{CPS19} uses techniques from tropical geometry. The goal of this work is to give new realizations in terms of brick polyhedra. The following is a purely combinatorial definition.

\begin{definition}[$\nu$-Associahedron \cite{CPS19}]\label{defasso}
The \emph{$\nu$-associahedron} is a polytopal complex induced by an arrangement of tropical hyperplanes, whose poset of faces (ordered by containment) is anti-isomorphic to the poset of interior faces of the $\nu$-Tamari complex (that is the poset of interior faces ordered by reverse containment).
\end{definition}

\begin{corollary}[\cite{ceballos_sweakorderII}] \label{inner}
The faces of the $\nu$-associahedron are in correspondence with pairs $(T, A)$, where $T$ is a $\nu$-tree and $A$ is a subset of its ascents. The dimension of the face corresponding to~$(T,A)$ is the cardinality $|A|$.
\end{corollary}

\begin{example}[$\nu$-Associahedron for $\nu=ENEEN$]\label{exENEENasso}
    Consider the $\nu$-subword complex for~$\nu=ENEEN$, the $\nu$-associahedron is shown in Figure~\ref{ENEENasso}, and its edge graph is the Hasse diagram of the $\nu$-Tamari lattice. The interior face $I_1$ illustrated in Figure~\ref{ENEENasso} corresponds to the orange line segment, while the interior face $I_2$ corresponds to the red pentagon. Note that $I_2 \subseteq I_1$, but the face corresponding to $I_1$ is contained in the face corresponding to $I_2$. The containment poset of interior faces is reversed.

    \begin{figure}[h]
    \centering
    \includegraphics[width=0.6\textwidth]{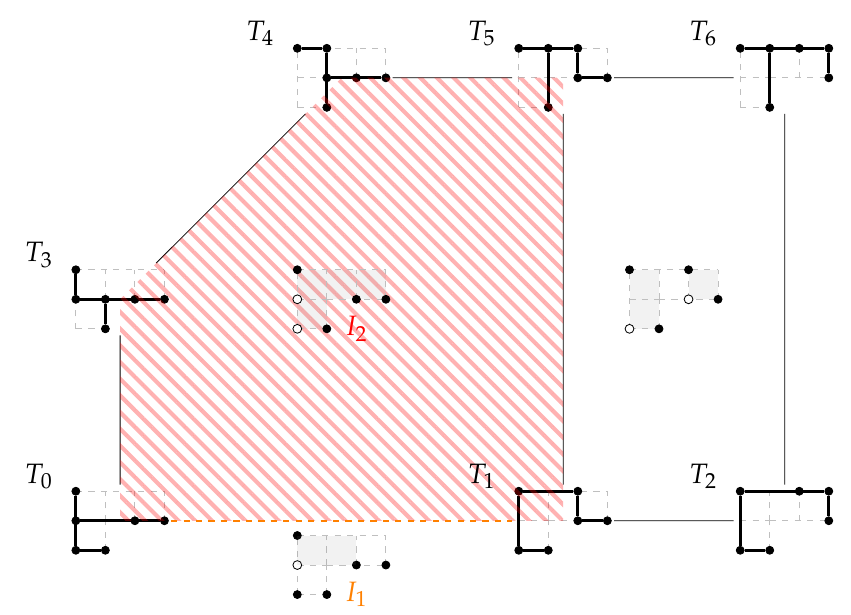} 
    \caption{$\nu$-Associahedron for $\nu=ENEEN$.}
    \label{ENEENasso}
    \end{figure}

\end{example}

\section{The $\nu$-brick polyhedron}
In this section, we introduce the $\nu$-brick polyhedron $\mathcal{B}(Q_\nu, w_\nu)$ and provide a useful tool (Corollary \ref{brickvectorofnutree}) to compute the brick vector $b(T)$ of a $\nu$-tree $T$.

\begin{definition}[$\nu$-Brick Polyhedron]
    For a lattice path $\nu$ the \emph{$\nu$-brick polyhedron} $\mathcal{B}(Q_\nu, w_\nu)$ is defined as the brick polyhedron of the $\nu$-subword complex~$\mathcal{SC}(Q_\nu, w_\nu)$. 
\end{definition}

\begin{convention}\label{conv}
We use the following convention to draw a pipe dream $P(T)$, a $\nu$-tree $T$ and the Ferrers diagram simultaneusly in the same figure. 
We shift the Ferrers diagram slightly in the direction $-(\epsilon, \epsilon)$ for some small $\epsilon > 0$. This is illustrated in Figure~\ref{fig_convpipe}. 
We say a pipe is above (or below) a lattice point in the Ferrers diagram if this is the case in the shifted figure. Additionally, we consider a pipe to be above a point if it lies entirely to the left of the point. This can be seen in Figure \ref{fig_conv}.
\end{convention}
\begin{figure}[h]
    \centering
    \includegraphics[width=0.15\textwidth]{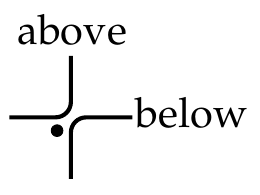}
    \hspace{1cm}
    \raisebox{-3.3mm}{\includegraphics[width=0.15\textwidth]{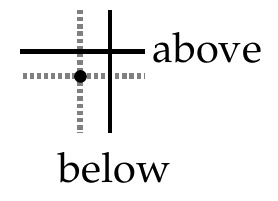}}
    \hspace{1cm}
    \raisebox{-4mm}{\includegraphics[width=0.15\textwidth]{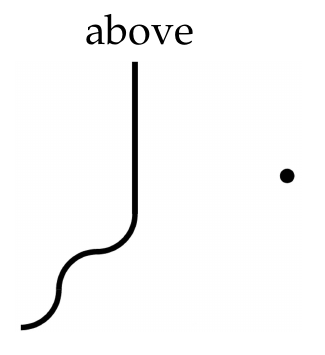}}
    \caption{Illustration of Convention~\ref{conv}.}
    \label{fig_conv}
\end{figure}

\begin{proposition}\label{weightfunctionnutree}
    For a $\nu$-tree $T$ the weight function of the corresponding facet in~$\mathcal{SC}(Q_\nu, w_\nu)$ is given by 
    $$\omega(T,k)= \sum  \limits _ {\text{$i$: pipe $i $ is above $p_k$}} e_i$$ 
    for $k\in [m]$, where $Q_\nu$ is of length $m$ and $p_k$ is the lattice point in the Ferrers diagram corresponding to the $k$th letter in $Q$.
\end{proposition}

\begin{corollary}\label{brickvectorofnutree}
        For a $\nu$-tree $T$ the $i$th entry of the corresponding brick vector is given by $$\text{$b(T)_i=-\# $ lattice points in the Ferrers diagram below pipe $i$.}$$
\end{corollary}

\begin{remark}
    This description is essentially the same as the classical description of brick vectors in type $A$, by counting bricks below pseudolines in a sorting network~\cite{bpofsn}.
\end{remark}

\begin{example}[$\nu$-Brick Polyhedron for $\nu=ENEEN$]\label{exENEENbp}
The word $Q_\nu$ and element $w_\nu$ are $Q_\nu=(s_3,s_2,s_1,s_4,s_3,s_2,s_4,s_3,s_5,s_4)$, and $w_\nu=s_3s_2s_3s_4$. By Corollary \ref{brickvectorofnutree}, the brick vectors are obtained by counting the points below the pipes. For instance, the computation for $T_4$ from Example \ref{exENEENasso}, using Convention~\ref{conv}, is shown in Figure~\ref{fig_convpipe}.
For instance, the $3$rd entry $b(T_4)_3=-5$ because there are $5$ lattice points in the Ferrers diagrams below pipe~$3$, as shown in~\Cref{fig_convpipe} (right). 
Counting lattice points below each of the six pipes, we get $b(T_4)=-(10,8,5,7,1,0)$.

\end{example}
\begin{figure}[h]
    \centering
\includegraphics[width=0.6\textwidth]{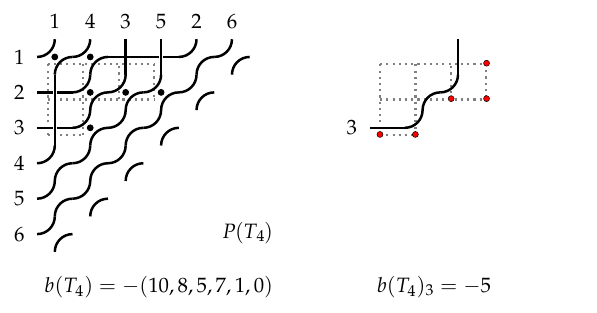}
    \caption{Computation of the brick vector for the $\nu$-tree $T_4$, by counting lattice points below pipes in the corresponding pipe dream.}
    \label{fig_convpipe}
\end{figure}

The brick vectors of all the $\nu$-trees from Example~\ref{exENEENasso} are:
\begin{align*}
    b(T_0) & = -(10,9,6,5,1,0) & b(T_4) & = -(10,8,5,7,1,0) \\
    b(T_1) & = -(10,7,8,5,1,0) & b(T_5) & = -(10,7,6,7,1,0) \\
    b(T_2) & = -(10,6,8,5,2,0) & b(T_6) & = -(10,6,6,7,2,0)  \\
    b(T_3) & = -(10,9,5,6,1,0)
\end{align*}
The $\nu$-brick polyhedron $\mathcal{B}(Q_\nu, w_\nu)\subseteq \mathbb{R}^6$ is of dimension $4$. However, to get a feeling about how it looks like, we can remove the letters $s_1$ and $s_5$ from $Q_\nu$. They are contained in every facet (are non-flippable) and give rays in the brick polyhedron. If we call $\widetilde{Q_\nu}$ the resulting word $\widetilde{Q_\nu}=(s_3,s_2,s_4,s_3,s_2,s_4,s_3,s_4)$ then the brick polyhedron $\mathcal{B}(\widetilde{Q_\nu}, w_\nu)$ is of dimension $3$ and is illustrated in Figure~\ref{ENEENcomp} (Left). The Bruhat cone is given by~$\mathcal{C}^+(w_\nu, \operatorname{Dem}(Q_\nu))= \operatorname{cone}\{\alpha_4, s_3(\alpha_4)\}$.

\section{A geometric realization via brick polyhedra}
The goal of this paper is to show that the $\nu$-associahedron can be geometrically realized as the complex of bounded faces of the $\nu$-brick polyhedron. 

The following is our main result, which can be observed from Examples \ref{exENEENasso} and \ref{exENEENbp}, illustrated in Figure~\ref{ENEENcomp}.

\begin{theorem}\label{main_theorem}
The $\nu$-associahedron is geometrically realized as the polytopal complex of bounded faces of the $\nu$-brick polyhedron $\mathcal{B}(Q_\nu, w_\nu)$. In other words, the poset of bounded faces of~$\mathcal{B}(Q_\nu, w_\nu)$ is anti-isomorphic to the poset of interior faces of the $\nu$-subword complex ($\cong\nu$-Tamari complex).
\end{theorem}
    \begin{figure}[h]
    \centering
    \includegraphics[width=0.9\textwidth]{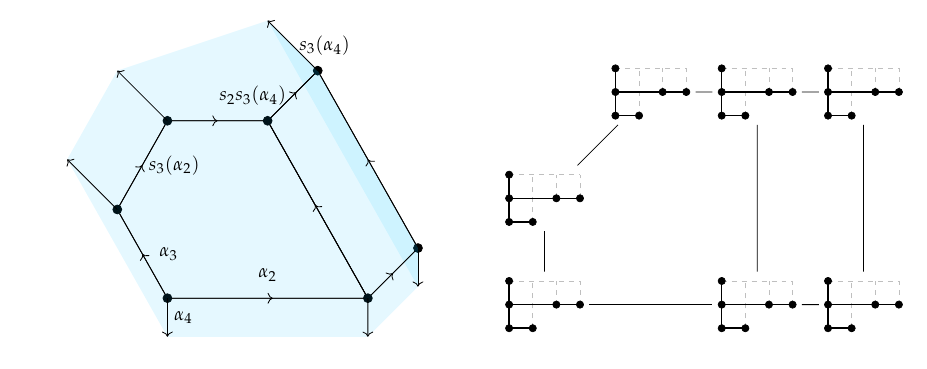} 
    \caption{Comparison of the $\nu$-brick polyhedron and  $\nu$-associahedron for $\nu=ENEEN$.}
    \label{ENEENcomp}
    \end{figure}

Our approach to establish this result involves the following steps:

\begin{enumerate}
\item Enhance the understanding of the faces of \textit{general} brick polyhedra (\Cref{facesall}).
\item Characterize the bounded faces of the $\nu$-brick polyhedron (\Cref{thm_main_bounded_faces}).
\item Analyze the poset of bounded faces of the $\nu$-brick polyhedron (Proof of Theorem~\ref{main_theorem} in Section \ref{finalsec}).
\end{enumerate}

\subsection{Faces of brick polyhedra} \label{sec_faces_brick_polyhedra}
In order to prove~\Cref{main_theorem}, it is useful to have a better understanding of the faces of brick polyhedra in general. For this purpose, we use a notion of modified brick polyhedra. 

\begin{definition}[Modified Bruhat Cone $\mathcal{C}^{I,+}$]
We denote by $\mathcal{SC}^I(Q,w)$ the set of all facets in the subword complex $\mathcal{SC}(Q,w)$ that contain a given face $I \in \mathcal{SC}(Q,w)$:
\begin{align*}
\mathcal{SC}^I(Q,w):=\{J \in \mathcal{SC}(Q,w) : I \subseteq J\}.
\end{align*}

For $J\in \mathcal{SC}^I(Q,w)$, the modified root configuration $R$$^I(J)$ is given by
\begin{align*}
\text{ $R$}^I(J):= \{ r(J,j) \mid j \in J \setminus I \}.
\end{align*}

 We define the modified Bruhat cone $\mathcal{C}^{I,+}$ as
\begin{align*}
\mathcal{C}^{I,+} := \bigcap_{J \in \mathcal{SC}^I(Q,w)} \operatorname{cone} \text{ $R$}^I(J).
\end{align*}

\end{definition}

\begin{definition}[Modified Brick Polyhedron $\mathcal{B}^I(Q, w)$]
For a face $I\in \mathcal{SC}(Q,w)$ of a non-empty subword complex, the \emph{modified brick polyhedron}~$\mathcal{B}^I(Q,w)$ is the polyhedron
\begin{center}
$\mathcal{B}^I(Q,w):= \text{conv}\{b(J)\mid$ J facet of $\mathcal{SC}(Q,w) \text{ and } I \subseteq J\}+\mathcal{C}^{I,+}$.
\end{center}
\end{definition}

\begin{example}\label{example3d}
For $\nu = ENEEN$ and using the labeling shown in Figure \ref{fig_labelENEEN}, the modified brick polyhedron $\mathcal{B}^{I_2}(Q, w)$ for $I_2=\{3,4,7,9\}$ is the bounded pentagon in Figure \ref{fig_brickpolyhedronENEEN}, while $\mathcal{B}^{I_3}(Q_\nu, w_\nu)$, where $I_3 = \{1, 2, 3, 9\}$, is not a face of the $\nu$-brick polyhedron $\mathcal{B}(Q_\nu, w_\nu)$ since it consists of a brick vector and two rays extending to infinity, as illustrated by the red region in Figure~\ref{fig_brickpolyhedronENEEN}. For $I_4 = \{3, 7, 8, 9\}$, the modified brick polyhedron $\mathcal{B}^{I_4}(Q_\nu, w_\nu)$ is a slice of the $\nu$-brick polyhedron $\mathcal{B}(Q_\nu, w_\nu)$, as illustrated by the green region in Figure~\ref{fig_brickpolyhedronENEEN}.
\end{example}

\begin{figure}[h]
    \centering
    \includegraphics[width=0.2\textwidth]{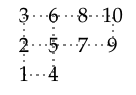}
    \caption{Grid for labeling of $\mathcal{SC}(Q_\nu, w_\nu)$ for $\nu=ENEEN$.}
    \label{fig_labelENEEN}
\end{figure}

\begin{figure}[h!]
    \centering
    \includegraphics[width=0.55\textwidth]{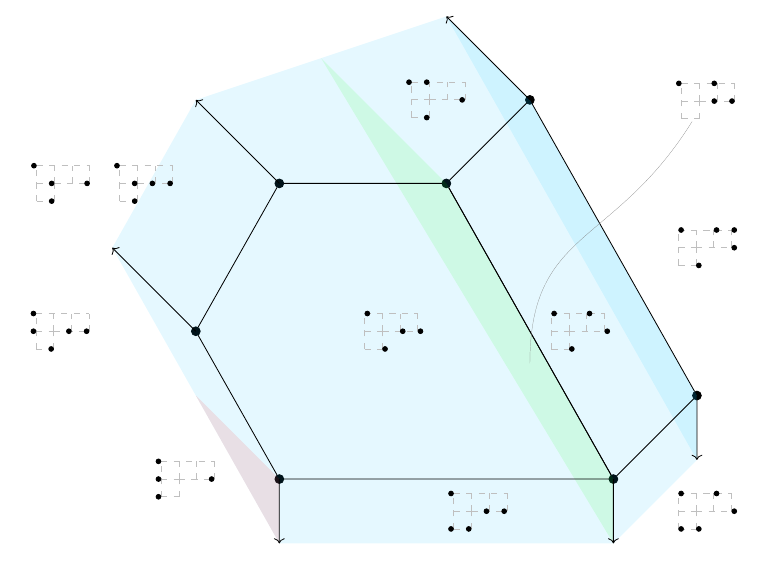}
    \caption{$\mathcal{B}(Q_\nu, w_\nu)$ for $\nu= ENEEN$.}
    \label{fig_brickpolyhedronENEEN}
\end{figure}

Moving forward, our objective is to establish the following result.

\begin{proposition}\label{facesall}
Every face $F$ of the brick polyhedron $\mathcal{B}(Q, w)$ is of the form $\mathcal{B}^I(Q, w)$ for some face $I \in \mathcal{SC}(Q,w)$. 
\end{proposition}


This proposition states that every face of a brick polyhedron $\mathcal{B}(Q, w)$ is a modified brick polyhedron $\mathcal{B}^I(Q, w)$; however, we remark that not every $\mathcal{B}^I(Q, w)$ is a face of~$\mathcal{B}(Q, w)$, as shown in Example~\ref{example3d}. The Proposition was essentially proved in \cite{bpolyhedra} using the description of faces via linear functionals, see~\cite[Remark~4.11]{bpolyhedra} and~\cite[Corollary~3.24 and Proposition~4.6]{bpolyhedra}. Our approach is a bit different and is based on~\Cref{faceprop}, cf. \cite[Proposition~4.6 and Remark~4.11]{bpolyhedra}.

For a face $F$ of $\mathcal{B}(Q,w)$, we denote by 
\[
V_F:=\operatorname{span} 
\{ p-q \mid p,q\in F\}
\] 
the vector space spanned by~$F$,
and by 
\[
\widetilde{V}_F:=
F+V_F= 
\{ p+v \mid p\in F, v\in V_F\}
\] 
the smallest affine space containing $F$.

It is essential to develop certain ideas before we can prove Proposition~\ref{facesall}.

\subsubsection{The minimal element and characterization of vertices in a face $F$}

First, we demonstrate that each face $F$ of the brick polyhedron $\mathcal{B}(Q, w)$ has a ``minimal element''. To accomplish this, let $\eta \in V$ be a vector such that $\langle \eta, \alpha \rangle > 0$ for $\alpha \in \Phi^+$ and $\langle \eta, \alpha \rangle < 0$ for $\alpha \in \Phi^-$. This vector $\eta$ can be regarded as a linear functional $\eta: V \rightarrow \mathbb{R}$ defined by $x \mapsto \langle \eta, x \rangle$, satisfying~$\eta(\alpha) \neq 0$ for all $\alpha \in \Phi$.

\begin{lemma}[Minimal element of a face] \label{minimalelement}
For every face $F$ of the brick polyhedron~$\mathcal{B}(Q,w)$, there exists a unique vertex $b(J_{F,\min}) \in F$, corresponding to some facet $J_{F,\min}$, that minimizes the linear functional $\eta$.
\end{lemma}

\begin{proof}
Let $F_\eta$ be the sub-face of $F$ that minimizes the linear functional $\eta$. We aim to demonstrate that $F_\eta$ consists of only one point. Otherwise, $F_\eta$ would contain a (possibly unbounded) edge in the direction of $\alpha$ for some $\alpha \in \Phi$. 

Let $p$ and $q$ be two points on this edge such that $q = p + \alpha$ for some $\alpha \in \Phi$. Then, $\langle \eta, q \rangle = \langle \eta, p \rangle + \langle \eta, \alpha \rangle$. Since $p$ and $q$ are minimizing $\eta$, it follows that $\langle \eta, \alpha \rangle = 0$, contradicting $\langle \eta, \alpha \rangle \neq 0$ for all $\alpha \in \Phi$.
\end{proof}

Our second step is to characterize the vertices within a face $F$ of the brick polyhedron~$\mathcal{B}(Q, w)$. 
The following Observation \ref{rprop} is a direct consequence of the definition of the reflection $r_\beta(\alpha) = \alpha - 2 \frac{\langle \alpha, \beta \rangle}{\langle \beta, \beta \rangle} \beta$ along the hyperplane orthogonal to the root $\beta$.

\begin{observation}\label{rprop}
Let $U$ be a finite-dimensional subspace of a vector space.
\begin{itemize}
    \item If $\alpha, \beta \in U$ then $r_\alpha(\beta)\in U $.
    \item If $\alpha \notin U$ and $\beta \in U$, then $r_\alpha(\beta) \not\in U $.
\end{itemize}

\end{observation}

\begin{lemma}[Characterization of the vertices in a face $F$]\label{ifflemma}
Let $F$ be a face of the brick polyhedron $\mathcal{B}(Q,w)$, and $J\in \mathcal{SC}(Q,w)$ be a facet such that the brick vector $b(J)\in F$. 
Define $J^F:=\{j \in J :$ r$(J,j) \in V_F\}$ and let~$I = J \setminus J^F$. For any facet $J' \in \mathcal{SC}(Q,w)$, we have the following equivalence:
\begin{align*}
I \subseteq J' \quad \text{if and only if} \quad b(J') \in \widetilde{V}_F.
\end{align*}
\end{lemma}

\begin{proof}
$\Rightarrow$: Assume $I \subseteq J'$. Note that there exists a sequence of flips $J = J_0 \overset{j_0}{\rightarrow} J_1 \overset{j_1}{\rightarrow} \ldots \overset{j_{\ell-1}}{\rightarrow} J_\ell = J'$ that never flips an element in $I$. This is due to the connectedness of the flip graph of a subword complex, especially for $\mathcal{SC}(Q_{[m]\setminus I},w)$, where $Q$ is of length $m$.

Now the following holds: 
\begin{enumerate}
\item r$(J_k,j) \notin V_F \text{ if } j\in  I$  \label{zwei}
\item r$(J_k,j) \in V_F \text{ if } j\in J_k \setminus I$ \label{eins}
\end{enumerate}

To prove this, it is enough to examine only one flip. Without loss of generality, take the first flip. Let $J_1=J_0\setminus \{j_0\} \cup \{j_0'\}$. 
By \cite[Lemma 3.3 (2)]{bpofssc}, we have 
\begin{empheq}[left={r(J_1,j)=\empheqlbrace}]{align*}
    & s_\beta(r(J_0,j))                  && \text{   if}\ \text{min}(j_0,j_0')<j \leq max(j_0,j_0') \\
    & r(J_0,j)                && \text{   otherwise}
  \end{empheq}

  for $\beta:= $r$(J_0,j_0)$. Since $\beta  \in V_F$, Observation \ref{rprop} implies properties (1) and (2) hold for $r(J_1,j)$.
  Now, note that $b(J_1) = b(J_0) + c_{0,1}$r$(J_0,j_0)$, for some constant $c_{0,1}$. Since r$(J_0,j_0) \in V_F$, we have $b(J_1) \in \widetilde{V}_F$. 
  
  Applying the same argument several times for the rotations in the sequence yields $b(J_\ell)= b(J') \in \widetilde{V}_F$.

$\Leftarrow$: Let $q = b(J') \in \widetilde{V}_F$ and $b(J_{F, \text{min}})$ be the minimal element by Lemma \ref{minimalelement}. We will show that $J'$ can be connected to $J_{F, \text{min}}$ by a sequence of decreasing flips $J' = J'_0 \overset{j'_0}{\rightarrow} J'_1 \overset{j'_1}{\rightarrow} \ldots \overset{j'_{\ell-1}}{\rightarrow} J'_{\ell'} = J_{F, \text{min}}$, with $r(J'_k, j'_k) \in V_F$ for all $1 \leq k < \ell'$. To demonstrate this, we provide a construction of a flip-sequence. Let $p_{\text{min}} := b(J_{F, \text{min}})$. By the local cone property \cite[Theorem 4.4]{bpolyhedra}, 

\begin{align*}
p_{\text{min}} - q=
b(J_{F, \text{min}}) - b(J')=
\sum_{j'\in J'} a_{j'} r(J', j') 
\in V_F,
\end{align*}
where $a_{j'} \geq 0$ for all $j' \in J'$. Let $\eta_F$ be a linear functional that is minimal at the face~$F$. Taking the inner product with $\eta_F$, we obtain:
\begin{align*}
0 = \sum_{j' \in J'} a_{j'} \langle r(J', j'), \eta_F \rangle 
\end{align*}
but $\langle r(J', j'), \eta_F \rangle =0$ if $r(J', j')\in V_F$ and $\langle r(J', j'), \eta_F \rangle >0$ otherwise. So, $a_{j'} = 0$ if~$r(J', j') \not\in V_F$. Therefore,
\begin{align*}
b(J_{F, \text{min}}) - b(J) = \sum_{\substack{j' \in J' \\ r(J', j') \in V_F}} a_{j'} r(J', j').
\end{align*}

\begin{figure}[h!]
    \centering
\includegraphics[width=0.3\textwidth]{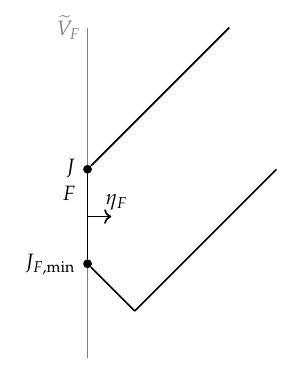}
    \caption{Normal vector $\eta_F$.}
\end{figure}

Since $J_{F, \text{min}}$ minimizes the linear functional $\eta$, the inner product with $\eta$ satisfies
\begin{align*}
\langle \eta, b(J_{F, \text{min}}) - b(J) \rangle < 0.
\end{align*}
If all $r(J', j') \in R(J) \cap V_F$ were positive roots, $\langle \eta, b(J_{F,\text{min}}) - b(J) \rangle > 0$ would hold. Thus, at least one $r(J', j')$ must be a negative root, $r(J', j_0') \in \Phi^-\cap V_F$. By \cite[Lema 3.3 (2)]{bpolyhedra} $j_0' \in J'$ is decreasingly flippable. Flipping $j_0'$ in $J$, we obtain a new facet $J_1'$
such that $b(J_1')\in \widetilde{V}_F$,
because $b(J_1')=b(J_0')+c_0r(J_0',j_0')$ for some positive constant $c_0$. 
If $J_1'=J_{F,\text{min}}$ then we are done. If not, we can repeat the process and find a sequence $J' \overset{j'_0}{\rightarrow} J'_1 \overset{j'_1}{\rightarrow} \ldots \overset{j'_{\ell'-1}}{\rightarrow} J_{F, \text{min}}$ such that $r(J'_k, j'_k) \in V_F$ for~$1 \leq k < \ell'$. 

The similar statement holds for the facet $J$. Thus, going through $J_{F, \text{min}}$, we can find a sequence $J = J_1 \overset{j_0}{\rightarrow} J_1 \overset{j_1}{\rightarrow} \ldots \overset{j_{\ell-1}}{\rightarrow} J_\ell = J'$ such that $r(J_k, j_k) \in V_F$ for $1 \leq k < \ell$. 
By properties \eqref{zwei} and \eqref{eins}, we are never flipping an element in~$I$ such that $j_k \notin I$.  Therefore, $I \subseteq J'$ holds. 
\end{proof}

\subsubsection{Proof of \Cref{facesall}}
Now are ready to prove~\Cref{facesall}. It is a direct consequence of the following more explicit result, cf.~\cite[Proposition 4.6 and Remark 4.11]{bpolyhedra}.

\begin{proposition}\label{faceprop}
Consider a face $F$ of the brick polyhedron $\mathcal{B}(Q,w)$, let $J \in \mathcal{SC}(Q,w)$ be a facet such that the brick vector $b(J)\in F$ and $J^F:=\{j \in J :$ r$(J,j) \in V_F\}$. Then, the following hold:
\begin{enumerate}
\item for $I= J \setminus J^F$, we have $F=\mathcal{B}^I(Q,w)$ and \label{p1}
\item for $J' \in \mathcal{SC}(Q,w)$ a facet, we have $I \subseteq J'$ if and only if $b(J') \in F$\label{p2}.
\end{enumerate}
\end{proposition}

\begin{proof}
By Lemma \ref{ifflemma}, $I \subseteq J'$ if and only if $b(J')\in F$, so \eqref{p2} follows. The vertices of~$F$ are the brick vectors $b(J')$ for $I \subseteq J'$, and the local cone at $b(J')$ inside $F$ is the intersection of the local cone of $b(J')$ in the brick polyhedron $\mathcal{B}(Q,w)$ with $\widetilde{V}_F$. This is precisely the local cone at $b(J')$ in the modified brick polyhedron $\mathcal{B}^I(Q,w)$. Thus, $F=\mathcal{B}^I(Q,w)$.
\end{proof}

 \subsection{Some faces of $\nu$-brick polyhedra}
In this subsection, we will outline some conditions under which the modified brick polyhedron $\mathcal{B}^I(Q_\nu, w_\nu)$ is a face of the $\nu$-brick polyhedron $\mathcal{B}(Q_\nu, w_\nu)$. The following notation will be used throughout our discussion.

\begin{definition}\label{conee}
For a $\nu$-tree $T$ we denote by $\beta_t:=r(T,t) = e_i - e_j$ the root associated to a node $t\in T$. We label $t$ by $ij$ for convenience and define the cone 
$$C(T) := \{ x \in V \mid \langle x, \beta_t \rangle > 0 \text{ for all } t \in T \}.$$
Given a subset of nodes~$M \subseteq T$ we denote the $\nu$-tree with $M$ marked by $T_M$, and define 
$$C(T_M) := \{x \in V \mid \langle x, \beta_t \rangle > 0 \text{ for all } t \in T \setminus M \text{ and } \langle x, \beta_t \rangle = 0 \text{ for all } t \in M \}.$$ 
An example is shown in Figure~\ref{markedex}.
\end{definition}
\begin{figure}[h!]
\centering
\includegraphics[width=\textwidth]{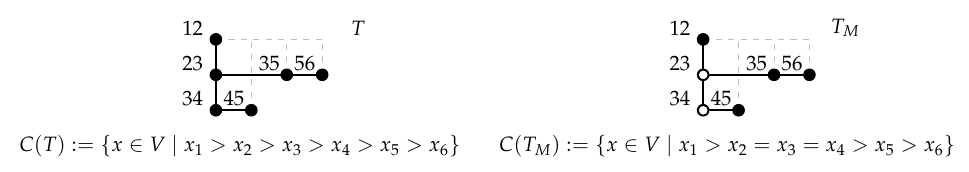}
\caption{Left: $\nu$-tree $T$, Right: $\nu$-tree with marked points $T_M$.}
\label{markedex}
\end{figure}
\begin{remark}\label{rema}
The result \cite[Theorem 4.17]{lofap} asserts that all the facets (pipe dreams) of the subword complex associated with the $\nu$-Tamari lattice are acyclic. This condition is equivalent to the system of inequalities having a solution. 
Indeed, the cone~$C(T)$ is non-empty if and only if the contact graph of the pipe dream $P(T)$ is acyclic. This acyclic property holds for all $\nu$-trees.
\end{remark}

\begin{lemma}\label{exlf}
Let $I=T \setminus A$, where $T$ is a $\nu$-tree and $A\subseteq T$ a subset of ascents, and let $\beta_t:=r(T,t)$ for $t \in T$. There exists a linear functional $f$ such that 
\begin{align}\label{ee}
f(\beta_a)=0\text{ for }a \in A\text{ and }f(\beta_t)>0\text{ for }t \in T \setminus A.
\end{align}
\end{lemma}

The proof idea of Lemma \ref{exlf} relies on Definition~\ref{conee}, by showing that the cone $C(T_A)\neq\emptyset$, and is exemplified in Example \ref{exproofidea}.

\begin{example}\label{exproofidea}
Consider $\nu = EEN$. All three $\nu$-trees $T$, $T'$, $T''$ and the defining inequalities of their corresponding cones $C(T_1),C(T_2),C(T_3)$ are shown in Figure~\ref{smallex}.
Note that $C(T_1),C(T_2)$ and $C(T_3)$ are non-empty.

Figure~\ref{marked} (Left) shows a marked $\nu$-tree $T_A$ together with the defining equalities and inequalities of $C(T_A)$. Here $A=\{a\}$ consists of the unique ascent $a\in T$. The figure also shows the marked $\nu$-tree $T'_{A'}$ where $T\smallsetminus\{a\}=T'\smallsetminus\{a'\}$. 
Our strategy to prove that $C(T_A)\neq \emptyset$ is to pick two points $p\in C(T)$ and $p'\in C(T')$, and show that there is a point $q$ in the line segment connecting them such that $q\in C(T_A)$. 
\end{example}

\begin{figure}[h]
\centering
\includegraphics[width=\textwidth]{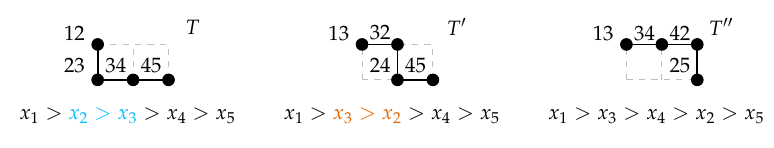}
    \caption{All three $\nu$-trees $T$, $T'$, $T''$ for $\nu = EEN$, and defining inequalities for $C(T)$, $C(T')$, $C(T'')$.}
    \label{smallex}
\end{figure}

\begin{figure}[h]
\centering
\includegraphics[width=\textwidth]{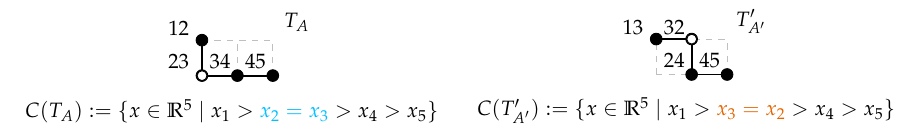}
\caption{Marked $\nu$-trees $T_A$ and $T'_{A'}$ for $\nu=EEN$.}
\label{marked}
\end{figure}

\begin{proof}[Proof of Lemma \ref{exlf}]
We proceed by induction on the number of marked ascents $k:=|A|$. The statement holds for $k = 0$ by \cite[Theorem 4.17]{lofap}, see Remark~\ref{rema}. Suppose $k \geq 1$ and that the statement holds for $k - 1$. Let $I=T\setminus A$, where $T$ is a $\nu$-tree and $A\subseteq T$ is a subset of ascents with $|A|=k$, and let $\beta_t=$$r$$(T,t)$, $t\in T$ as above.

Let $\bar{a}\in A$ be the northeast-most node in $A$ (i.e. no other node in $A$ is located northeast of $\bar{a}$), and let $T'=T\setminus \{\bar{a}\}\cup \{\bar{a}'\}$ be the $\nu$-tree obtained from $T$ by rotating $\bar{a}$.

\begin{figure}[h]
  \centering
\begin{tabular}{cc}
    \begin{tikzpicture}[scale=1.8]

    \draw (1,-1) node [scale=0.6, circle, draw,fill=black,anchor=center]{};
    \draw (1,-2) node [scale=0.6, circle, draw,fill=black,anchor=center]{};
    \draw (2,-2) node [scale=0.6, circle, draw,fill=black,anchor=center]{};

\draw[black, thick] (0.4,-2) -- (0.8,-2) arc (-90:0:0.2) -- (1,-1.2) arc (180:90:0.2) -- (2.4,-1);
\draw[black, thick] (0.4,-1) -- (0.8,-1) arc (-90:0:0.2) -- (1,-0.6);
\draw[black, thick]  (1,-2.6)--(1,-2.2) arc (180:90:0.2) -- (1.8,-2) arc (-90:0:0.2) -- (2,-0.6);
\draw[black, thick]  (2,-2.6)--(2,-2.2) arc (180:90:0.2) -- (2.4,-2) ;

    \node[black, scale=1] at (0.2,-2) {$i_2$};
    \node[black, scale=1] at (0.2,-1) {$i_1$};
    \node[black, scale=1] at (1,-2.8) {$i_3$};
    \node[black, scale=1] at (2,-2.8) {$i_4$};
    \node[black, scale=1] at (0.5,-2.5) {$T$};
    \node[black, scale=1] at (1.3,-0.8) {$\notin A$};
    \node[black, scale=1] at (1.2,-1.8) {$\bar{a}$};
    \node[black, scale=1] at (2.3,-1.8) {$\notin A$};
\end{tikzpicture}

&

    \begin{tikzpicture}[scale=1.8]
    \draw (1,-1) node [scale=0.6, circle, draw,fill=black,anchor=center]{};
    \draw (2,-1) node [scale=0.6, circle, draw,fill=black,anchor=center]{};
    \draw (2,-2) node [scale=0.6, circle, draw,fill=black,anchor=center]{};

\draw[black, thick] (0.4,-1) -- (0.8,-1) arc (-90:0:0.2) -- (1,-0.6);
\draw[black, thick]  (2,-2.6)--(2,-2.2) arc (180:90:0.2) -- (2.4,-2) ;
\draw[black, thick]  (1,-2.6)--(1,-1.2) arc (180:90:0.2) -- (1.8,-1) arc (-90:0:0.2) -- (2,-0.6);
\draw[black, thick] (0.4,-2) -- (1.8,-2) arc (-90:0:0.2) --(2,-1.2)  arc (180:90:0.2)--(2.4,-1);

    \node[black, scale=1] at (0.2,-2) {$i_2$};
    \node[black, scale=1] at (0.2,-1) {$i_1$};
    \node[black, scale=1] at (1,-2.8) {$i_3$};
    \node[black, scale=1] at (2,-2.8) {$i_4$};
    \node[black, scale=1] at (0.5,-2.5) {$T'$};
    \node[black, scale=1] at (1.3,-0.8) {$\notin A$};
    \node[black, scale=1] at (2.2,-0.8) {$\bar{a}'$};
    \node[black, scale=1] at (2.3,-1.8) {$\notin A$};
\end{tikzpicture}
\end{tabular}
\caption{Structure of the $\nu$-trees $T$ and $T'$.}
\label{figproof}
\end{figure}
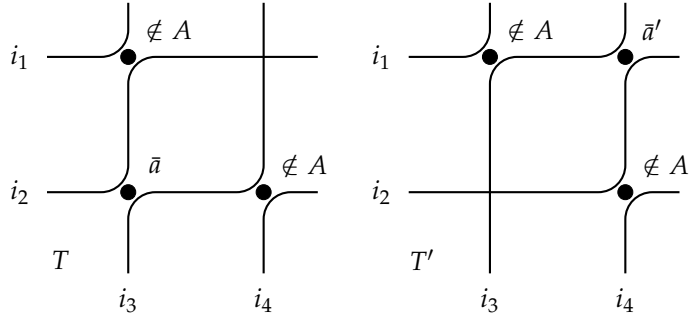

Let $i_1,i_2,i_3,i_4$ be the pipes passing through the nodes of $T$ involved in the rotation of~$\bar{a}\in T$, as illustrated in Figure \ref{figproof}. Let $\bar{A}=A\setminus \{\bar{a}\}$, then $\bar{A} \subseteq T$ is a subset of ascents of $T$ with $|\bar{A}|=k-1$ and $\bar{A}\subseteq T'$ is also a subset of ascents of $T'$. By induction hypothesis, the inequalities and equalities defining $C(T_{\bar{A}})$ and $C(T'_{\bar{A}})$ have solutions. Let 
\begin{align*}
p=(p_1,...,p_\ell) \in C(T_{\bar{A}}), \\
p'=(p'_1,...,p'_{\ell'}) \in C(T'_{\bar{A}}).
\end{align*}
Note that all defining inequalities and equalities defining $C(T_{\bar{A}})$ and $C(T'_{\bar{A}})$ coincide, except for three:
\begin{align*}
x_{i_1}>x_{i_2}>x_{i_3}>x_{i_4} \text{ in } C(T_{\bar{A}}), \\
x_{i_1}>x_{i_3}>x_{i_2}>x_{i_4} \text{ in } C(T'_{\bar{A}}).
\end{align*}
On the other hand, the defining inequalities and equalities for $C(T_{A})$ are all other defining inequalities for $C(T_{\bar{A}})$ and $C(T'_{\bar{A}})$ together with
\begin{align*}
x_{i_1}>x_{i_2}=x_{i_3}>x_{i_4} \text{ in } C(T_{\bar{A}}).
\end{align*}
Let $\bar{D}$ be defined by the other inequalities and equalities together with
\begin{align*}
x_{i_1}>x_{i_2}\text{, }x_{i_3}>x_{i_4}.
\end{align*}
Note that we do not put any condition between $x_{i_2}$ and $x_{i_3}$ in $\bar{D}$, while
\begin{align*}
x_{i_2}>x_{i_3} \text{ in } C(T_{\bar{A}}), \\
x_{i_3}>x_{i_2} \text{ in } C(T'_{\bar{A}}), \\
x_{i_2}=x_{i_3} \text{ in } C(T_{A}).
\end{align*}
We want to show that $ C(T_{A})\neq \emptyset$. For this, note that $C(T_{\bar{A}}) \subseteq \bar{D}$ and $C(T'_{\bar{A}}) \subseteq \bar{D}$. Therefore 
\begin{align*}
q(t)=(1-t)p+tp' \in \bar{D} \text{ for all } t \in [0,1].
\end{align*}
Since
\begin{align*}
q_{i_2}(0)=p_{i_2}>p_{i_3}=q_{i_3}(0) \text{ and } q_{i_2}(1)=p'_{i_2}<p'_{i_3}=q_{i_3}(1) .
\end{align*}
Then, there must be some $r\in [0,1]$ such that
\begin{align*}
q_{i_2}(r)=q_{i_3}(r).
\end{align*}
Since all other defining inequalities of $C(T_{A})$ are satisfied for $q(r)$, then
\begin{align*}
q(r)\in C(T_{A}).
\end{align*}
Thus, $C(T_{A})$ has a solution as wanted.
\end{proof}

\begin{remark}
Note that Lemma \ref{exlf} is not necessarily true if \( A \subseteq T \) is not a subset of ascents, as demonstrated in Example~\ref{exgbigger}.
\end{remark}

\begin{example}\label{exgbigger}
Let us continue Example \ref{example3d} and consider the marked $\nu$-trees $T_{M_2}$, $T_{M_3}$, and $T'_{M_4}$ as illustrated in Figure \ref{choice} and let $I_2=T \setminus M_2$, $I_3=T \setminus M_3$, $I_4=T' \setminus M_4$.

For $T_{M_2}$, we obtain the conditions $x_1 > x_2 = x_3 = x_4 > x_5 > x_6$, which has a solution, and the modified brick polyhedron $\mathcal{B}^{I_2}(Q_\nu, w_\nu)$ is a face.
For $T_{M_3}$, we obtain the conditions $x_1 > x_2 > x_3 > x_4 = x_5 > x_6$ with $x_3 = x_5$, which has no solution, and the modified brick polyhedron $\mathcal{B}^{I_3}(Q_\nu, w_\nu)$ is not a face.
For $T'_{M_4}$, we obtain the conditions $x_1 > x_3 = x_4 = x_5$ and $x_3 > x_2 > x_5 > x_6$, which has no solution, and the modified brick polyhedron $\mathcal{B}^{I_4}(Q_\nu, w_\nu)$ is not a face.

\end{example}

\begin{figure}[h!]
\centering
\includegraphics[width=\textwidth]{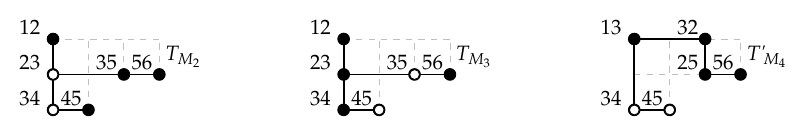}
\caption{Marked $\nu$-trees for $\nu=ENEEN$.}
\label{choice}
\end{figure}

\begin{proposition}\label{isaface}
Let $I=T \setminus A$, where $T$ is a $\nu$-tree and $A\subseteq T$ a subset of ascents. Then the modified brick polyhedron~$\mathcal{B}^I(Q_\nu, w_\nu)$ is a face of the $\nu$-brick polyhedron $\mathcal{B}(Q_\nu, w_\nu)$.
\end{proposition}

\begin{proof}
Take $f$ as in Lemma \ref{exlf} and let $F$ be the face of the brick polyhedron~$\mathcal{B}(Q_\nu, w_\nu)$ minimizing~$f$. Since for every $\nu$-tree $\tilde{T}$ we have
\begin{align*}
b(\tilde{T})-b(T)=\sum_{t\in T}c_t r(T,t) \text{ for some } c_t\geq0.
\end{align*}
Then $f(b(\tilde{T})) \geq  f(b(T))$, with equality satisfied when $b(\tilde{T})-b(T) \in V_F$, the vector space spanned by $\{r(T,a)\mid a \in A\}$. In particular $b(T)\in F$ and by Lemma~\ref{faceprop}~\eqref{p2}
\begin{align*}
b(\tilde{T}) \in F \Longleftrightarrow T\setminus A=I \subseteq \tilde{T}.
\end{align*}

Now $F=\mathcal{B}(Q_\nu, w_\nu)\cap \widetilde{V_F}$, and the local cone at point $b(\tilde{T})$ inside $F$ is the intersection of the local cone of $b(\tilde{T})$ in $\mathcal{B}(Q_\nu, w_\nu)$ with $\widetilde{V_F}$. This is equal to the local cone of $b(\tilde{T})$ in $\mathcal{B}^I(Q_\nu, w_\nu)$. Therefore, $F=\mathcal{B}^I(Q_\nu, w_\nu)$.
\end{proof}

\subsection{Bounded faces of $\nu$-brick polyhedra}
The goal of this section is to show that the faces in Proposition \ref{isaface} are exactly the bounded faces of the $\nu$-brick polyhedron (\Cref{thm_main_bounded_faces}).

\subsubsection{The spherical and root independent property}

We denote by $Q_{\nu,I}$ the word obtained from $Q_\nu$ by deleting the letters with positions in~$I$, and consider the corresponding subword complex $\mathcal{SC}(Q_{\nu,I},w_\nu)$. The purpose of this section is to show that $\mathcal{SC}(Q_{\nu,I},w_\nu)$ is spherical and root independent when~$I=T\setminus A$ for a $\nu$-tree~$T$ and~$A\subseteq T$ a subset of ascents. As a consequence, $\mathcal{SC}(Q_{\nu,I},w_\nu)$ is realized as the polar of the brick polytope $\mathcal{B}(Q_{\nu,I},w_\nu)$, which is combinatorially isomorphic to $\mathcal{B}^I(Q_{\nu},w_\nu)$. Before doing this, we need some preliminaries.

\begin{theorem}[\cite{bpofssc}] \label{pilaudstump}
If $\mathcal{SC}(Q,w_\circ)$ is root independent, then it is realized by the polar of the brick polytope $\mathcal{B}(Q, w_\circ)$.
\end{theorem}

\begin{lemma}[\cite{doolittle}, \cite{ceballos_and_labbe}]\label{fact}
 Every spherical subword complex $\mathcal{SC}(Q,w)$ is isomorphic to a (spherical) subword complex of the form $\mathcal{SC}(\Tilde{Q},w_\circ)$.
\end{lemma}

\begin{remark}[Completing]\label{completing}
For any spherical subword complex $\mathcal{SC}(Q,w)$, there exists a word ${w'}$, such that $w{w'} = w_\circ$ with $\ell(w)+\ell(w')=\ell(w_\circ)$, by completing~$w$ to~$w_\circ$. Defining~${Q'}=Q {w'}$, the subword complexes $\mathcal{SC}({Q},w)$ and $\mathcal{SC}({Q'},w_\circ)$ are isomorphic and the brick polytope~$\mathcal{B}(Q,w)$ is just a translation of~$\mathcal{B}(Q',w_\circ)$.
\end{remark}

The following result is assumed/mentioned in \cite{bpofssc} but not explicitly written down. We include it here with proof for completeness.
\begin{corollary}[{\cite{bpofssc}}]\label{cors}
Every spherical, root independent subword complex $\mathcal{SC}(Q,w)$, where $w$ is not necessary equal to $w_\circ$, is realized by the polar of the brick polytope $\mathcal{B}(Q,w)$.
\end{corollary}

\begin{proof}
By Lemma \ref{fact}, $\mathcal{SC}(Q,w) \cong \mathcal{SC}(Q',w_\circ)$ by completing, as in Remark \ref{completing} the facets and root configurations do not change, so $\mathcal{SC}({Q'},w_\circ)$ is root independent and we can apply Theorem~\ref{pilaudstump}. Moreover the brick polytope of $\mathcal{SC}(Q',w_\circ)$ is a translation of the brick polytope of $\mathcal{SC}(Q,w)$, so the Corollary holds.
\end{proof}

Our next goal is to prove the following proposition. 
\begin{proposition}\label{sphericalandrootindependent}
    Let $I=T \setminus A$, where $T$ is a $\nu$-tree and $A\subseteq T$ a subset of ascents, then~$\mathcal{SC}(Q_{\nu,I},w_\nu)$ is a spherical and root independent subword complex.
\end{proposition}

The proof follows in two steps. \Cref{lemmaspherical} shows the spherical property and~\Cref{proofA1} shows root independence.

\begin{lemma}[Spherical]\label{lemmaspherical}
For an interior face $I\in \mathcal{SC}(Q_\nu,w_\nu)$, the subword complex~$\mathcal{SC}(Q_{\nu, I},w_\nu)$ is realizable as the boundary of a polytope, hence it is spherical.
\end{lemma}

\begin{proof}
If $I$ is an interior face then the set of faces $\{J \in \mathcal{SC}(Q_\nu,w_\nu) : I \subseteq J\}$, ordered by reverse containment is the face poset of the face of the $\nu$-associahedron corresponding to $I$. But the reverse containment poset on $\{J \in \mathcal{SC}(Q_\nu,w_\nu) : I \subseteq J\}$ is isomorphic to the reverse containment poset of faces of $\mathcal{SC}(Q_{\nu, I},w_\nu)$. And by \cite[Proposition 5.16]{CPS19} the cells of the $\nu$-associahedron are known to be products of classical associahedra, in particular they are polytopes. And $\mathcal{SC}(Q_{\nu, I},w_\nu)$ can be realized as the boundary complex of the polar of the polytope of the corresponding cell. Hence $\mathcal{SC}(Q_{\nu,I},w_\nu)$ is polytopal, hence spherical.
\end{proof}

\begin{remark}\label{2kisses}
Observe, a single pipe cannot have more than two turns inside the Ferrers diagram $F_\nu$. Otherwise, there would exist at least three vertices, as shown in Figure~\ref{2turns} (Left) in the $\nu$-tree, but then the two red points would be $\nu$-incompatible, which contradicts the definition of a $\nu$-tree as a maximal set of $\nu$-compatible elements. 
\end{remark}

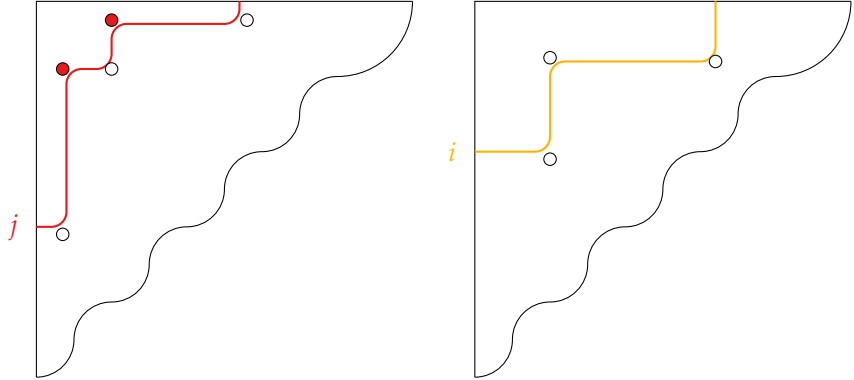
\begin{figure}[h!]
  \centering
\begin{tabular}{cc}
    \begin{tikzpicture}[scale=1]
    \draw (0,-5) -- (0,0) -- (5,0);
    
    \draw (0,-5) to[out=0,in=-90] (0.5,-4.5)
                 to[out=90,in=180] (1,-4)
                 to[out=0,in=-90] (1.5,-3.5)
                 to[out=90,in=180] (2,-3)
                 to[out=0,in=-90] (2.5,-2.5)
                 to[out=90,in=180] (3,-2)
                 to[out=0,in=-90] (3.5,-1.5)
                 to[out=90,in=180] (4,-1)
                 to[out=0,in=-90] (5,0);

\draw[red, thick] (0,-3) -- (0.2,-3) arc (-90:0:0.2) -- (0.4,-1.1) arc (180:90:0.2) -- (0.8,-0.9) arc (-90:0:0.2) -- (1,-0.5) arc (180:90:0.2) -- (2.5,-0.3)  arc (-90:0:0.2) --(2.7,0);

\node[red] at (-0.3,-3) {$j$};
 \draw (1,-0.9) node [scale=0.5, circle, draw,fill=white,anchor=center]{};

  \draw (0.35,-0.9) node [scale=0.5, circle, draw,fill=red,anchor=center]{};
 \draw (1,-0.25) node [scale=0.5, circle, draw,fill=red,anchor=center]{};
 \draw (0.35,-3.1) node [scale=0.5, circle, draw,fill=white,anchor=center]{};
 \draw (2.8,-0.25) node [scale=0.5, circle, draw,fill=white,anchor=center]{};
    \end{tikzpicture}
    
&
    
    \begin{tikzpicture}[scale=1]

    \draw (0,-5) -- (0,0) -- (5,0);
    
    \draw (0,-5) to[out=0,in=-90] (0.5,-4.5)
                 to[out=90,in=180] (1,-4)
                 to[out=0,in=-90] (1.5,-3.5)
                 to[out=90,in=180] (2,-3)
                 to[out=0,in=-90] (2.5,-2.5)
                 to[out=90,in=180] (3,-2)
                 to[out=0,in=-90] (3.5,-1.5)
                 to[out=90,in=180] (4,-1)
                 to[out=0,in=-90] (5,0);

\draw[orange, thick] (0,-2) -- (0.8,-2) arc (-90:0:0.2) -- (1,-1) arc (180:90:0.2) -- (3,-0.8) arc (-90:0:0.2) -- (3.2,0);

\node[orange] at (-0.3,-2) {$i$};
 \draw (1,-0.75) node [scale=0.5, circle, draw,fill=white,anchor=center]{};
 \draw (3.2,-0.8) node [scale=0.5, circle, draw,fill=white,anchor=center]{};
 \draw (1,-2.1) node [scale=0.5, circle, draw,fill=white,anchor=center]{};

    \end{tikzpicture}
  \end{tabular}
\caption{Left: A pseudoline with $3$ turns exhibiting red vertices in $\nu$-incompatible position, Right: A pseudoline with $2$ turns showing all vertices in $\nu$-compatible position.}
\label{2turns}
\end{figure}

\begin{lemma}[Root Independent]\label{proofA1}
Let $I$ be an interior face of $\mathcal{SC}(Q_\nu, w_\nu)$, $T$ a $\nu$-tree and $A\subseteq T$ a subset of ascents of $T$ such that $I=T \setminus A$. Then,
\begin{enumerate}
\item $\{$$r$$(T,a) : a \in A\}$ is linearly independent. \label{L1}
\item $\mathcal{SC}(Q_{\nu,I}, w_\nu)$ is root independent.\label{L2}
\end{enumerate}

\end{lemma}

\begin{proof}

To show \eqref{L1}, we proceed by induction on $n=|A|$. The statement is clear for $n=1$. Suppose that $n \geq 2$ and that the statement holds for $n-1$.

Let $a_2 \in A$ such that $r(T,a_2) = e_i - e_j$ with the smallest possible $i$. Since $a_2 \in A$ is an ascent, there is a node to the North and one to the East. By Remark~\ref{2kisses}, there cannot be a node to the left of $a_2$, as illustrated in Figure~\ref{pseudolines}. Furthermore, since $i$ is chosen to be minimal, $a_2$ is unique. Additionally, there can be no vertex on the line segment between the vertices $a_2$ and~$a_3$, following the same argument.

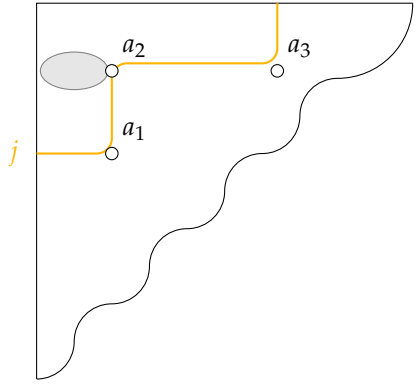
\begin{figure}[h]
  \centering
        \begin{tikzpicture}[scale=1]

    \draw (0,-5) -- (0,0) -- (5,0);
    
    \draw (0,-5) to[out=0,in=-90] (0.5,-4.5)
                 to[out=90,in=180] (1,-4)
                 to[out=0,in=-90] (1.5,-3.5)
                 to[out=90,in=180] (2,-3)
                 to[out=0,in=-90] (2.5,-2.5)
                 to[out=90,in=180] (3,-2)
                 to[out=0,in=-90] (3.5,-1.5)
                 to[out=90,in=180] (4,-1)
                 to[out=0,in=-90] (5,0);

\draw[orange, thick] (0,-2) -- (0.8,-2) arc (-90:0:0.2) -- (1,-1) arc (180:90:0.2) -- (3,-0.8) arc (-90:0:0.2) -- (3.2,0);

\node[orange] at (-0.3,-2) {$j$};
\fill[gray!20] (0.5,-0.9) ellipse (0.45 and 0.25);
 \draw[gray] (0.5,-0.9) ellipse (0.45 and 0.25);
 \draw (1,-0.9) node [scale=0.5, circle, draw,fill=white,anchor=center]{};
 \draw (3.2,-0.9) node [scale=0.5, circle, draw,fill=white,anchor=center]{};
 \draw (1,-2) node [scale=0.5, circle, draw,fill=white,anchor=center]{};

\node[black] at (1.3,-0.6) {$a_2$}; 
\node[black] at (3.5,-0.6) {$a_3$}; 
\node[black] at (1.3,-1.7) {$a_1$}; 
    \end{tikzpicture}
  \caption{No node in gray area.}
  \label{pseudolines}
\end{figure}

Now, suppose there exists a linear combination such that $\sum_{a\in A} c_a r(T,a)=0$. Since $r(T,a_2)=e_i-e_j$ with $i$ being minimal, there is no summand of the form $e_k-e_i$, hence $c_{a_2}=0$. Therefore $\sum c_a r(T,a)$ is a sum of $n-1$ summands and by induction hypothesis, we obtain $c_{a}=0$ for all $a \in A$. So \eqref{L1} follows.

To prove \eqref{L2} it is enough to find a facet of $\mathcal{SC}(Q_{\nu, I}, w_\nu)$ whose root configuration is linearly independent. Taking the facet $A \subseteq \mathcal{SC}(Q_{\nu,I}, w_\nu)$ and applying \eqref{L1} concludes~\eqref{L2}.
\end{proof}

\begin{corollary}\label{star1}
Let $I=T \setminus A$, where $T$ is a $\nu$-tree and $A\subseteq T$ a subset of ascents.
The subword complex $\mathcal{SC}(Q_{\nu, I}, w_\nu)$ is realized by the polar of the brick polytope $\mathcal{B}(Q_{\nu, I}, w_\nu)$. In particular~$\mathcal{B}(Q_{\nu, I}, w_\nu)$ is a bounded polytope.
\end{corollary}
\begin{proof}
By~\Cref{sphericalandrootindependent}, $\mathcal{SC}(Q_{\nu, I}, w_\nu)$ is spherical and root independent. Therefore, it is realized by the polar of $\mathcal{B}(Q_{\nu, I}, w_\nu)$ by Corollary \ref{cors}.
\end{proof}

\subsubsection{Characterization of bounded faces}
We are now ready to characterize the bounded faces of the $\nu$-brick polyhedron.

\begin{corollary}\label{thm_bounded_faces}
Let $T$ a $\nu$-tree, $A\subseteq T$ be a subset of ascents, and $I=T \setminus A$ be the corresponding interior face of $\mathcal{SC}(Q_{\nu}, w_\nu)$. Then, $\mathcal{B}^I(Q_{\nu}, w_\nu)$ is a bounded face of the $\nu$-brick polyhedron. Moreover, 
\begin{align*}
\mathcal{B}^I(Q_{\nu}, w_\nu) = \text{conv} \{ b(J)\mid I \subseteq J \text{ a facet of $\mathcal{SC}(Q_{\nu}, w_\nu)$}\}
\end{align*}
and its face poset is the reverse containment poset on the set  $\{J \in \mathcal{SC}(Q_\nu,w_\nu) : I \subseteq J\}$.
\end{corollary}

\begin{proof}
By \cite[Remark 4.11]{bpolyhedra} the brick polytope $\mathcal{B}(Q_{\nu, I}, w_\nu)$ and the brick polyhedron $\mathcal{B}^I(Q_{\nu}, w_\nu)$ share the following properties:
\begin{enumerate}
\item Their vertices are in correspondence 
$$J \in \mathcal{SC}(Q_{\nu, I}, w_\nu) \Longleftrightarrow I \cup J \in \mathcal{SC}^I(Q_{\nu}, w_\nu).$$
\item Their edge directions are the same, although they may have different lengths (roots in the root configuration and modified root configuration are the same).
\item Their local cones at the vertices are the same.
\end{enumerate}
In particular, $\mathcal{B}(Q_{\nu, I}, w_\nu)$ and $\mathcal{B}^I(Q_{\nu}, w_\nu)$ are combinatorially isomorphic, and in particular bounded (Corollary \ref{star1}). Since $\mathcal{SC}(Q_{\nu, I}, w_\nu)$ is realized by the polar of~$\mathcal{B}(Q_{\nu, I}, w_\nu)$ by Corollary~\ref{star1}, transforming the facets $J \in \mathcal{SC}(Q_{\nu, I}, w_\nu)$ to facets~$I \cup J \in \mathcal{SC}^I(Q_{\nu}, w_\nu)$ we obtain:
\begin{align*}
\mathcal{B}^I(Q_{\nu}, w_\nu) = \text{conv} \{b(J) \mid I \subseteq J \text{ a facet of $\mathcal{SC}(Q_\nu,w_\nu)$}\}.
\end{align*}
Furthermore, since the face poset of $\mathcal{B}(Q_{\nu,I}, w_\nu)$ is the reverse containing poset of faces $J\in \mathcal{SC}(Q_{\nu, I}, w_\nu)$, then adding $I$ to each face we obtain that the face poset of $\mathcal{B}^I(Q_{\nu}, w_\nu)$ is the reverse containment poset on the set $\{J' \in \mathcal{SC}(Q_\nu, w_\nu) \mid I \subseteq J'\}$ as desired.
\end{proof}

\begin{lemma}[Unique representation of bounded faces]\label{lemmaforuni}\label{bounded}
Let $T$ be a $\nu$-tree, $A\subseteq T$ a subset of ascents, $I=T \setminus A$ and $F=\mathcal{B}^I(Q_\nu, w_\nu)$ be the corresponding bounded face of the $\nu$-brick polyhedron $\mathcal{B}(Q_\nu, w_\nu)$. Then
\begin{enumerate}
\item $T=J_{F,\text{min}}$ and $A=J^F_{F,\text{min}}$. \label{u2}
\item If $F=\mathcal{B}^{I'}(Q_\nu, w_\nu)$ for $I'=T' \setminus A'$ where $T'$ is a $\nu$-tree and $A'\subseteq T'$ a subset of ascents, then $T=T'$, $A=A'$ and $I=I'$. \label{u3}
\end{enumerate}
\end{lemma}

\begin{proof}
Note that $J_{F,\text{min}}$ is the unique facet $I\subseteq J$ such that $j \in J \setminus I$ is increasingly flippable. Since $T$ satisfies this property, then $T=J_{F,\text{min}}$. Moreover, $J^F_{F,\text{min}} = J_{F,\text{min}}\setminus I = T \setminus I = A$. This proves property~\eqref{u2}.

Property \eqref{u3} follows from \eqref{u2}.
\end{proof}

\begin{corollary}\label{star3}
The bounded faces of the $\nu$-brick polyhedron $\mathcal{B}(Q_\nu, w_\nu)$  are exactly, the~$\mathcal{B}^I(Q_\nu, w_\nu)$, for $I=T\setminus A$, $T$ a $\nu$-tree and $A\subseteq T$ a subset of ascents.
\end{corollary}
\begin{proof}
By Corollary \ref{thm_bounded_faces}, if $T$ is a $\nu$-tree, $A\subseteq T$ a subset of ascents, and $I=T\setminus A$ then $\mathcal{B}^I(Q_\nu, w_\nu)$ is a bounded face of~$\mathcal{B}(Q_\nu, w_\nu)$. Now let $F$ be a bounded face of~$\mathcal{B}(Q_\nu, w_\nu)$, $T=J_{F, \text{min}}$ and $A=J^F_{F, \text{min}}$. Since~$F$ is bounded then every element of~$A$ is flippable in $T$, otherwise~$F$ would contain an infinite ray. Moreover, every~$a\in A$ is increasingly flippable in $T$, because $T$ is the minimal element of the face. The face $F=\mathcal{B}^I(Q_\nu, w_\nu)$ as desired.
\end{proof}

\begin{question}\label{question_faces_brick_polyhedra}
Can we characterize the sets $I$ for which the modified brick polyhedron~$\mathcal{B}^I(Q_\nu, w_\nu)$ is a face of the brick polyhedron~$\mathcal{B}(Q_\nu, w_\nu)$? Corollary \ref{star3} gives an answer for bounded faces but we do not know an answer in general. See Example \ref{example3d}.
\end{question}

The following theorem summarizes the results of this section. 

\begin{theorem}\label{thm_main_bounded_faces}
The bounded faces of the $\nu$-brick polyhedron $\mathcal{B}(Q_\nu, w_\nu)$  are exactly the modified brick polyhedra~$\mathcal{B}^I(Q_\nu, w_\nu)$, for $I=T\setminus A$, where $T$ is a $\nu$-tree and $A\subseteq T$ is a subset of ascents.
Moreover, 
\begin{align*}
\mathcal{B}^I(Q_{\nu}, w_\nu) = \text{conv} \{ b(J)\mid I \subseteq J \text{ a facet of $\mathcal{SC}(Q_{\nu}, w_\nu)$}\},
\end{align*}
and its face poset is the reverse containment poset on the set  $\{J \in \mathcal{SC}(Q_\nu,w_\nu) : I \subseteq J\}$.
\end{theorem}

\begin{proof}
    This following directly from~\Cref{thm_bounded_faces} and~\Cref{star3}.
\end{proof}

\subsection{The poset of bounded faces of $\nu$-brick polyhedra}\label{finalsec}

It remains to show that the poset of bounded faces of the $\nu$-brick polyhedron is anti-isomorphic to the poset of interior faces of the $\nu$-subword complex, as stated in Theorem~\ref{main_theorem}.

\begin{proof}[Proof of~\Cref{main_theorem}]\label{finalproof}
By Corollary \ref{star3}, the bounded faces of~$\mathcal{B}(Q_\nu, w_\nu)$ are exactly the $\mathcal{B}^I(Q_\nu, w_\nu)$ for the interior face $I=T\setminus A$ for some $\nu$-tree $T$ and $A\subseteq T$ a subset of ascents. We are going to show: If $I_1$, $I_2$ are interior faces of $\mathcal{SC}(Q_\nu, w_\nu)$ then 
\begin{align*}
I_1 \subseteq I_2 \Longleftrightarrow \mathcal{B}^{I_2}(Q_\nu, w_\nu)\subseteq \mathcal{B}^{I_1}(Q_\nu, w_\nu).
\end{align*}
Corollary \ref{thm_bounded_faces} implies that if $I$ is an interior face, then
\begin{align*}
\mathcal{B}^I(Q_\nu, w_\nu) = \text{conv}\{b(J)\mid I \subseteq J \text{ facet}\}.
\end{align*}
Therefore, if $I_1 \subseteq I_2$ then $\mathcal{B}^{I_2}(Q_\nu, w_\nu)\subseteq \mathcal{B}^{I_1}(Q_\nu, w_\nu)$. Indeed all faces of $\mathcal{B}^{I_1}(Q_\nu, w_\nu)$ are of the form $\mathcal{B}^{J}(Q_\nu, w_\nu)$ for $I_1 \subseteq J$. 

We are labelling faces of the brick polyhedron by $\mathcal{B}^{J}(Q_\nu, w_\nu)$, where $J= T \setminus A$, such a labelling is unique by Lemma \ref{lemmaforuni}. Now take a face $\mathcal{B}^{I_2}(Q_\nu, w_\nu)$ of~$\mathcal{B}^{I_1}(Q_\nu, w_\nu)$ then $\mathcal{B}^{I_2}(Q_\nu, w_\nu)= \mathcal{B}^{J}(Q_\nu, w_\nu)$ for some $I_1 \subseteq J$. By uniqueness, we have $I_2=J$ and $I_1 \subseteq I_2$.
\end{proof}

\begin{corollary}
The complex of bounded faces of the $\nu$-brick polyhedron is a realization of the $\nu$-associahedron.
\end{corollary}

\section{A projection} 
Since the dimension of the $\nu$-brick polyhedron is usually much higher than the dimension of the $\nu$-associahedron, it is interesting to study suitable projections to obtain figures in the appropriate dimension. In this section, we provide an elegant projection in the case where~$\nu$ has no two consecutive north steps. 

For convenience, we consider paths of the form~$\nu=(NE^{k_n}) \cdot \cdot \cdot (NE^{k_1})$, where \mbox{$k_i\geq 1$}. Notice that we are adding a north step $N$ at the beginning and some east steps at the end of the path, but this does not affect the combinatorics of the bounded components of the brick polyhedron. 
One can double check that the $\nu$-brick polyhedron $\mathcal{B}(Q_\nu,w_\nu) \subseteq \mathbb{R}^{n+2+\sum(k_i-1)}$.  
Since the first and last coordinates of  the brick vectors $b(T)$ are constant for every $\nu$-tree~$T$, we omit them, and write $\widetilde{b}(T)$ for the resulting vectors. 
Moreover, we denote by $\widetilde{\mathcal{B}}(Q_\nu,w_\nu)$ the result of omitting the first and last coordinates of the intersection of the brick polyhedron ${\mathcal{B}} (Q_\nu,w_\nu)$ with the affine subspace defined by the first and the last coordinates being those constant numbers. In particular, 
the bounded components of ${\mathcal{B}} (Q_\nu,w_\nu)$ are in correspondence with the bounded components of  $\widetilde{\mathcal{B}} (Q_\nu,w_\nu)$. 
After removing the first and the last coordinates, we have  
\[
\widetilde{\mathcal{B}} (Q_\nu,w_\nu)\subseteq \mathbb{R}^{n+\sum(k_i-1)}=\mathbb{R}^N
\]
where $N:=n+\sum(k_i-1)$.
We denote by $x_I:= \sum_{i\in I} x_i$.
Our projection uses the sets 
\( \widetilde{M}_1, \dots, \widetilde{M}_n \), which are defined recursively by setting \( \widetilde{M}_j \) to be the last \( (k_i - 1) \) elements of \( [N] \setminus \bigcup_{i=1}^{j-1} \widetilde{M}_i \), and let \( M_j = \widetilde{M}_j \cup \{ j \} \) for \( j \in [n] \).

\begin{definition}[Projection]
We define the projection 
\begin{align*}
    \pi_1: \mathbb{R}^{N} \longrightarrow \mathbb{R}^{n}, (x_1,...,x_{N})\mapsto (x_{M_1},...,x_{M_{n}})\in \mathbb{R}^{n} . 
\end{align*}
Next, let
\begin{align*}
    \pi_2: \mathbb{R}^{n} \rightarrow \mathbb{R}^{n-1},(x_{M_1},...,x_{M_{n}}) \mapsto (x_{M_1},x_{M_1}+x_{M_2},...,x_{M_1}+...+x_{M_{n-1}}).
\end{align*}
Finally, we define $\pi:\mathbb{R}^{N} \rightarrow \mathbb{R}^{n-1}$ as $\pi=\pi_2 \circ \pi_1$.
\end{definition}

\begin{theorem}\label{T1}
    Let $\nu=(NE^{k_n}) \cdot \cdot \cdot (NE^{k_1})$ with $k_i\geq 1$ (no consecutive north steps). Then the projection $\pi: \mathbb{R}^N \rightarrow \mathbb{R}^{n-1}$ of the bounded components of $\widetilde{\mathcal{B}} (Q_\nu,w_\nu)$ is a realization of the $\nu$-associahedron of the desired dimension.
\end{theorem}

Moreover, the coordinates of the projected vertices can be simply described as follows. Let $T_0$ be the minimal $\nu$-tree of the $\nu$-Tamari lattice. We denote by 
\[
y(T)=(y_1,\dots,y_{n-1}) := 
\pi (\widetilde{b}(T) ) - \pi (\widetilde{b}(T_0) ).
\]
We can think of these y-coordinates as the result of translating the projection of the bounded components of the $\nu$-brick polyhedron by the constant vector $- \pi (\widetilde{b}(T_0) )$. 
In particular, $y(T_0)=(0,\dots,0)$.

It turns out that these new coordinates can be described in a very simple and elegant combinatorial way. 
First, we label the horizontal lines of the Ferrers diagram determined by the path $\nu$ from $1$ to $n-1$, from top to bottom, omitting the top most row. We denote by~$P_i=P_i(T)$ the unique shortest path connecting the root to the left most node of $T$ on the $i$-th horizontal line, and let $\operatorname{area}( P_i )$ be the number of boxes left to $P_i$. Two examples are illustrated in~\Cref{canonical}. 
Magically, $y_i=\operatorname{area}(P_i)$.

\begin{figure}[h!]
    \centering
    \includegraphics[width=0.75\textwidth]{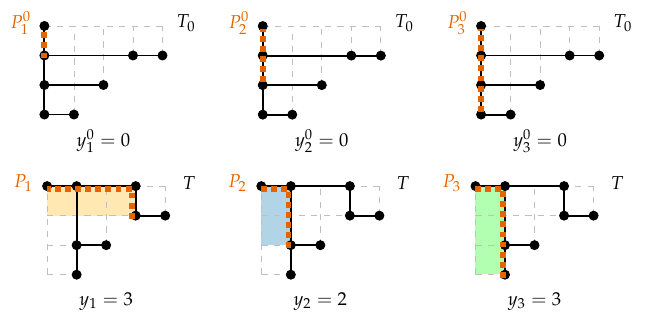} 
    \caption{The canonical coordinates $y(T_0)=(y_1^{0},y_2^{0},y_3^{0})=(0,0,0)$, and $y(T)=(y_1,y_2,y_3)=(3,2,3)$ for the two $\nu$-trees $T_0$ and $T$. The entry~$y_i(T)$ is the area (i.e. number of boxes to the left) of the path $P_i(T)$ connecting the root to the leftmost node of $T$ at level~$i$ (increasing from top to bottom).}    \label{canonical}
\end{figure}

\begin{theorem}\label{T2}
    Let $\nu=(NE^{k_n}) \cdot \cdot \cdot (NE^{k_1})$ with $k_i\geq 1$ (no consecutive north steps).    For a $\nu$-tree $T$ we have $y(T) =(y_1,...,y_{n-1})$, where $y_i = \operatorname{area} (P_i(T))$.
    These coordinates determine a realization of the $\nu$-associahedron. 
\end{theorem}

\begin{remark}
    As a consequence, we obtain that the projection $\pi: \mathbb{R}^N \rightarrow \mathbb{R}^{n-1}$ of the bounded components of $\widetilde{\mathcal{B}} (Q_\nu,w_\nu)$ is a translation of the canonical realization of the $\nu$-associahedron described by the first author in~\cite{Ceb24}. This may be regarded as a Loday-like realization of the $\nu$-associahedron, because in the classical case $\nu=(NE)^n$ both realizations are affinely equivalent. 
\end{remark}

\begin{example}
We continue Example \ref{exENEENasso} and Example \ref{exENEENbp}, but add a north step~$N$ at the beginning and an east step $E$ at the end of the path for convenience, obtaining $\nu=NENEENE$ (instead of $\nu=ENEEN$). The seven $\nu$-trees are shown in~\Cref{fig_ENEENproj}.

    \begin{figure}[!h]
    \centering
    \includegraphics[width=0.6\textwidth]{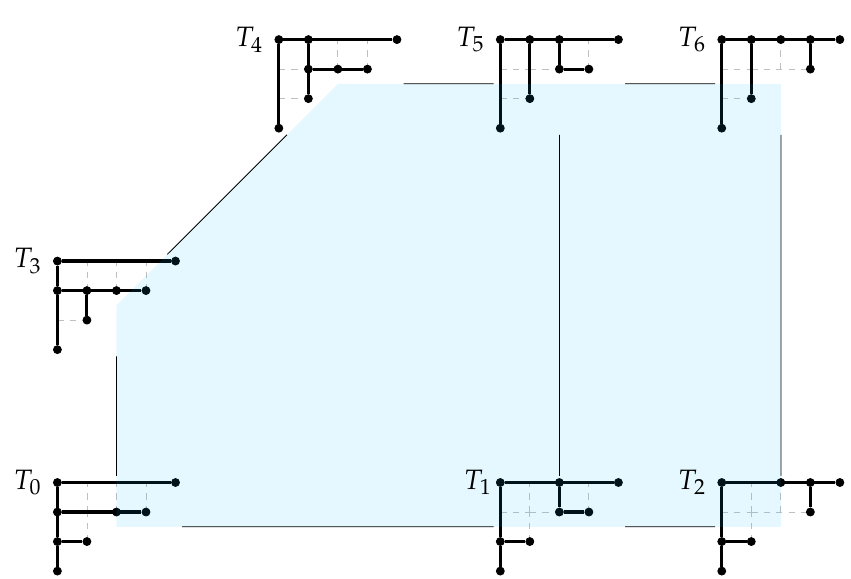} 
    \caption{Projection of the bounded components of $\widetilde{\mathcal{B}}(Q_\nu, w_\nu)$ for $\nu=NENEENE$.}
    \label{fig_ENEENproj}
    \end{figure}
\end{example}
The brick vectors are:
\begin{align*}
    b(T_0) & = -(12,11,8,7,2,0) & b(T_4) & = -(12,10,7,9,2,0) \\
    b(T_1) & = -(12,9,10,7,2,0) & b(T_5) & = -(12,9,8,9,2,0) \\
    b(T_2) & = -(12,8,10,7,3,0) & b(T_6) & = -(12,8,8,9,3,0)  \\
    b(T_3) & = -(12,11,7,8,2,0)
\end{align*}

Note that these are just translations of the brick vectors we computed in Example~\ref{exENEENbp} (because of the extra $N$ and $E$). The vectors $\widetilde{b}(T_i)\in \mathbb{R}^4$ are obtained by removing the first and last coordinates of $b(T_i)$. Furthermore, we have $M_1=\{1\}$, $M_2=\{2,4\}$ and $M_3=\{3\}$.  
The projected points are then:
\begin{align*}
    \pi (\widetilde{b}(T_0))  & = -(11,21) & \pi (\widetilde{b}(T_4)) & = -(10,19) \\
    \pi (\widetilde{b}(T_1)) & = -(9,21) & \pi (\widetilde{b}(T_5)) & = -(9,19) \\
    \pi (\widetilde{b}(T_2)) & = -(8,21) & \pi (\widetilde{b}(T_6)) & = -(8,19)  \\
    \pi (\widetilde{b}(T_3)) & = -(11,20)
\end{align*}

These coordinates coincide with the vertex coordinates of the canonical realization by \cite{Ceb24}, up to a translation by $\pi (\widetilde{b}(T_0) )=-(11,21)$:
\begin{align*}
    y(T_0) & = (0,0) & y(T_4) & = (1,2) \\
    y(T_1) & = (2,0) & y(T_5) & = (2,2) \\
    y(T_2) & = (3,0) & y(T_6) & = (3,2)  \\
    y(T_3) & = (0,1)
\end{align*}

The resulting projection is illustrated in Figure \ref{fig_ENEENproj}.

\begin{example}
For $\nu=NENENEENE=(NE^1)(NE^1)(NE^2)(NE^1)$, we obtain $N=n+\sum (k_i-1)=4+(0+1+0+0)=5$. 
Furthermore, $M_1=\{1\}$, $M_2=\{2,5\}$, $M_3=\{3\}$, and $M_4=\{4\}$. So we group coordinates $2$ and $5$ of $\widetilde{b}(T)$ together. 
In order to illustrate how the projection works, let us consider the two $\nu$-trees $T$ and~$T_0$ in Figure~\ref{canonical}.
We obtain:

\begin{align*}
    b(T) & = -(17, 13, 13, 9, 13, 2, 0) & b(T_0) & = -(17, 16, 12, 10, 10, 2, 0) \\
    \widetilde{b}(T) & = -(13, 13, 9, 13, 2) & \widetilde{b}(T_0) & = -(16, 12, 10, 10, 2) \\
    \pi_1(\widetilde{b}(T)) & = -(13, 15, 9, 13) & \pi_1(\widetilde{b}(T_0)) & = -(16, 14, 10, 10)  \\
    \pi(\widetilde{b}(T)) & = -(13, 28, 37) & \pi(\widetilde{b}(T_0)) & = -(16, 30, 40) \\
\end{align*}

The difference between the corresponding projected brick vectors is 
\[
y(T)=(y_1,y_2,y_3)=\pi(\widetilde{b}(T)) - \pi(\widetilde{b}(T_0)) = (3,2,3). 
\] 
As we can see from Figure~\ref{canonical}, the entry $y_i=\operatorname{area}(P_i)$ counts the number of boxes left to the path $P_i$ connecting the root to the left most node of $T$ at level $i$.

Although the $\nu$-brick polyhedron ${\mathcal{B}} (Q_\nu,w_\nu)\subseteq \mathbb R^7$, the projection of its bounded components lies in $\mathbb R^3$ and is illustrated in Figure~\ref{fig_ENENEENproj}.
    \begin{figure}[h]
    \centering
    \includegraphics[width=0.37\textwidth]{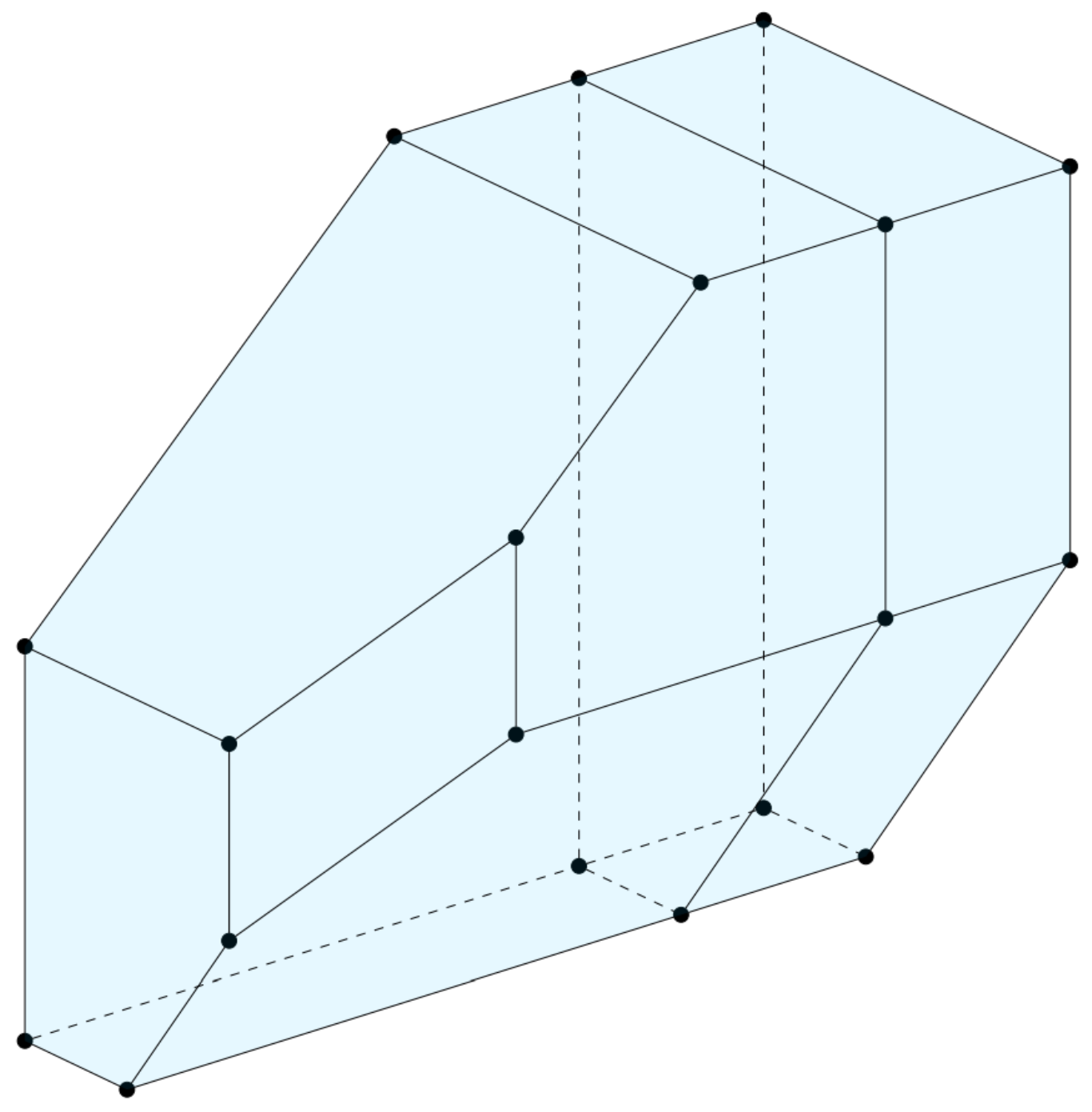} 
    \includegraphics[width=0.4\textwidth]{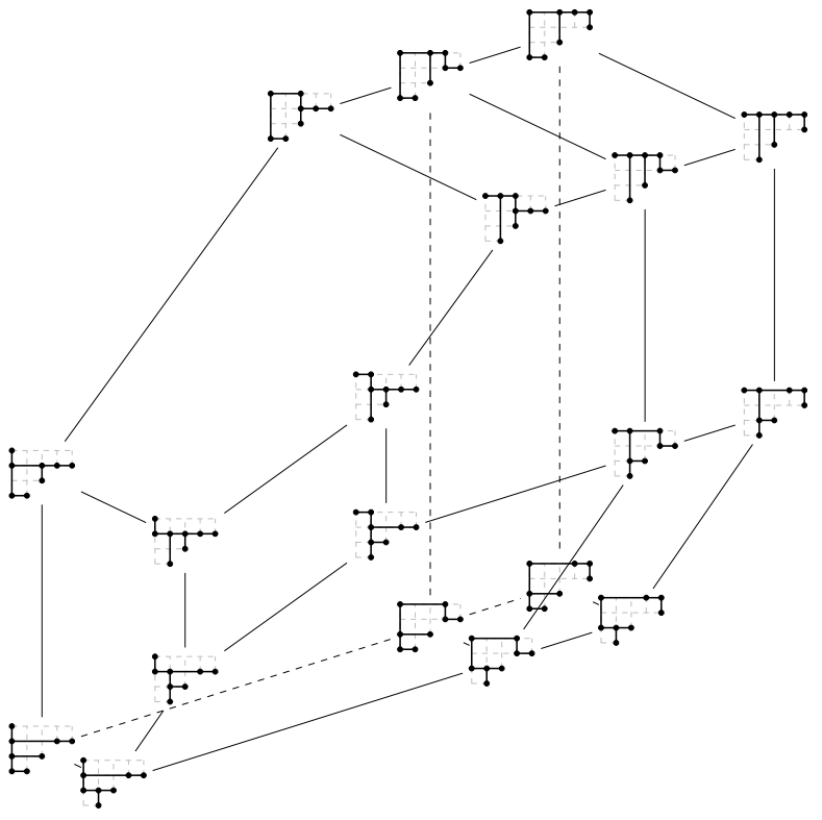}
    \caption{Left: Projection of the bounded components of $\widetilde{\mathcal{B}}(Q_\nu, w_\nu)$, Right: $\nu$-associahedron.}
    \label{fig_ENENEENproj}
    \end{figure}
\end{example}

\begin{example}
    Figure~\ref{fig_nu_associahedra} shows the projections of $\nu$-brick polyhedra to the canonical realization of $\nu$-associahedra for three different lattice paths: 
    \begin{itemize}
        \item $\nu= (NE)(NE)(NEE)(NE)$
        \item $\nu= (NE^2)^4$ (Fuss-Catalan case)
        \item $\nu= (NE^3)^4$ (Fuss-Catalan case)
    \end{itemize}
    Animations of these figures can be found in~\cite{matthiasWebsite}.
\end{example}

\begin{figure}[h]
    \centering
    \includegraphics[width=0.8\textwidth]{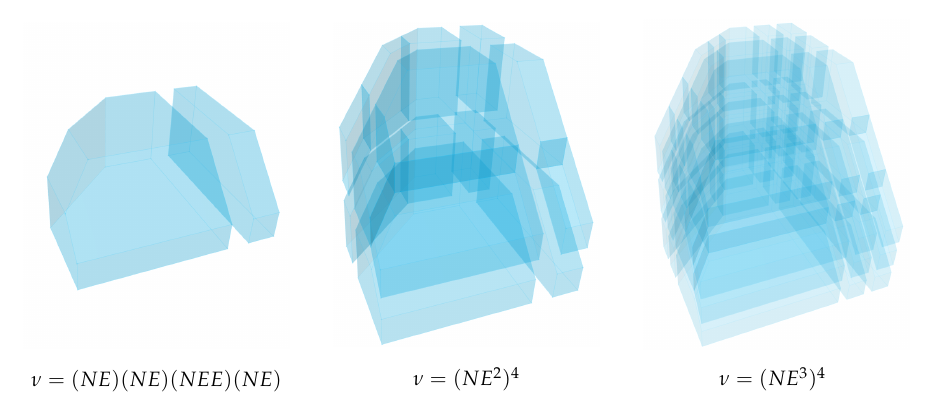} 
    \caption{Some $\nu$-associahedra obtained as projections of $\nu$-brick polyhedra.}
    \label{fig_nu_associahedra}
\end{figure}

\subsection{Proof of Theorems \ref{T1} and \ref{T2}}
Since Theorem \ref{T2} provides the explicit coordinates, which, up to translation, coincide with the known canonical realization of the $\nu$-associahedron in~\cite{Ceb24}, it is sufficient to prove Theorem  \ref{T2}.

\begin{proof}[Proof of \Cref{T2}]
Observe that the dimension of the maximal polytopal part and the projected points align precisely. Let $T$ be a $\nu$-tree and $T'=T\setminus \{b\} \cup \{b'\}$ be a rotation as in Figure \ref{proofspecialcase}.

\begin{figure}[h]
  \centering
  \begin{tabular}{cc}

    \begin{tikzpicture}[scale=1]

    \fill [gray!20] (2,-3)--(2,-2)--(3,-2)--(2,-3);

    
    \draw [color=black, dotted] (0,-3)--(2,-3);
    \draw (1,-1) node [scale=0.3, circle, draw,fill=black,anchor=center]{};

    \draw (1,-2) node [scale=0.3, circle, draw,fill=black,anchor=center]{};
    \draw (2,-2) node [scale=0.3, circle, draw,fill=black,anchor=center]{};

    \node[orange, scale=0.8] at (-0.2,-2) {$i$};
    \node[red, scale=0.8] at (-0.2,-4) {$j$};
    \node[black, scale=0.8] at (-0.2,-3) {$k$};

    \node[black, scale=1, right=5pt] at (0.75,-1.55) {$\text{area}_{i,j}$};
    \node[black, scale=1, left=5pt] at (3.1,-2.5) {$A_1$};
    \node[black, scale=1, left=5pt] at (2.5,-4) {$T$};
\draw[orange, thick] (0,-2) -- (0.8,-2) arc (-90:0:0.2) -- (1,-1.2) arc (180:90:0.2) -- (3,-1) arc (-90:0:0.2) -- (3.2,0);
\draw[red, thick] (0,-4) -- (0.8,-4) arc (-90:0:0.2) -- (1,-2.2) arc (180:90:0.2) -- (1.8,-2) arc (-90:0:0.2) -- (2,0);
    \node[black, scale=1, left=5pt] at (1,-1) {$a$};
    \node[black, scale=1, left=5pt] at (1,-2) {$b$};
    \node[black, scale=1, right=5pt] at (2,-2) {$c$};
\end{tikzpicture}

&

        \begin{tikzpicture}[scale=1]

    \fill [gray!20] (2,-3)--(2,-2)--(3,-2)--(2,-3);

    
    \draw [color=black, dotted] (0,-3)--(2,-3);
    \draw (1,-1) node [scale=0.3, circle, draw,fill=black,anchor=center]{};

    \draw (2,-1) node [scale=0.3, circle, draw,fill=black,anchor=center]{};
    \draw (2,-2) node [scale=0.3, circle, draw,fill=black,anchor=center]{};

    \node[orange, scale=0.8] at (-0.2,-2) {$i$};
    \node[red, scale=0.8] at (-0.2,-4) {$j$};
    \node[black, scale=0.8] at (-0.2,-3) {$k$};
    \node[black, scale=1, left=5pt] at (1,-1) {$a$};
    \node[black, scale=1, left=5pt] at (2.7,-0.7) {$b'$};
    \node[black, scale=1, right=5pt] at (2,-2) {$c$};
    \node[black, scale=1, right=5pt] at (0.75,-1.55) {$\text{area}_{i,j}$};
    \node[black, scale=1, left=5pt] at (3.1,-2.5) {$A_1$};
    \node[black, scale=1, left=5pt] at (2.5,-4) {$T'$};
\draw[orange, thick] (0,-2) -- (1.8,-2) arc (-90:0:0.2) -- (2,-1.2) arc (180:90:0.2) -- (3,-1) arc (-90:0:0.2) -- (3.2,0);
\draw[red, thick] (0,-4) -- (0.8,-4) arc (-90:0:0.2) -- (1,-1.2) arc (180:90:0.2) -- (1.8,-1) arc (-90:0:0.2) -- (2,0);
\end{tikzpicture}
\end{tabular}
\caption{Structure of $\nu$-trees $T$ and $T'$.}
\label{proofspecialcase}
\end{figure}
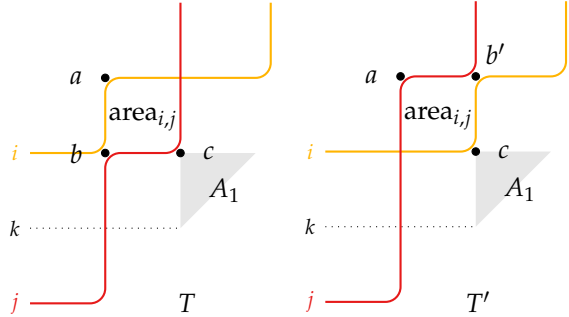

Let $c(T)=(c_1,...,c_d)$ and $c(T')=(c'_1,...,c'_d)$ be the coordinates of the $\nu$-trees $T$ and $T'$ in the canonical realization. We need to show
\begin{align*}
c(T)-c(T') = \pi (b(T')) - \pi(b(T)).
\end{align*}

Consider the pipe dream representation of the $\nu$-trees $T$ and $T'$.
The translation vector from the canonical realization to the projection is then $\pi(b(T_0))$, where $T_0$ is the minimal $\nu$-tree. Suppose that node $b\in T$ touches pipes $i$ and~$j$, and let~$k$ be the lower horizontal level of a descendant node of $b$ in the tree $T$, see Figure~\ref{proofspecialcase}. Denoting the number of boxes covered by a rotation as $\text{area}_{i,j}$ we obtain the following in the canonical realization:
\begin{align*}
c'_\ell= \begin{cases} 
c_\ell +\text{area}_{i,j} & \text{ for $i \leq \ell \leq k$ } \\
c_\ell & \text{otherwise}
\end{cases}.
\end{align*}
We need to show that this formula holds for the projected points. Let 
\begin{align*}
y = (y_1,...,y_d) = \pi(b(T))\text{, }y' = (y_{1'},...,y_{d'}) = \pi(b(T'))\text{, }\\
x = (x_1,...,x_{N-2}) = b(T)\text{, }x' = (x_{1'},...,x_{N-2}') = b(T').
\end{align*}
Then $x_i'=x_i+\text{area}_{i,j}$, $x_j'=x_j-\text{area}_{i,j}$ and $x'_\ell = x_\ell$ otherwise. This is because, by Corollary \ref{brickvectorofnutree}, $x_i$ is the negative of the number of lattice points in the Ferrers diagram that are below pipe $i$. When we make the rotation, this number of points is reduced by $\text{area}_{i,j}$. Similarly, the number of points below pipe $j$ increases by $\text{area}_{i,j}$. The number of points below the other pipes remains constant.

Now it is not hard to see that $j\in M_{k+1}$. If there is a node below $b$ in $T$ then~$j=k+1 \in M_{k+1}$, if not then $j$ is the diagonal level plus $1$ of the point $s\in S$ that is in the $k$-th horizontal line, so $j \in M_{k+1}$. Therefore, since
\begin{align*}
y_\ell=(x_{M_1}+...+x_{M_\ell})\text{ and }y_\ell'=(x'_{M_1}+...+x'_{M_\ell})
\end{align*}
then 
\begin{align*}
y'_\ell= \begin{cases} 
y_\ell +\text{area}_{i,j} & \text{ for $i \leq \ell \leq k$ } \\
y_\ell & \text{otherwise}
\end{cases}.
\end{align*}
This finishes the proof.
\end{proof}

\section*{Acknowledgement}
This work was funded by the Austrian Science Fund (FWF), project numbers  10.55776/P33278 and 10.55776/I5788.

\bibliographystyle{alpha}
\bibliography{sn-bibliography}  

\begin{thebibliography}{vBCFdSG}

\bibitem[BCCP25]{lofap}
Nantel Bergeron, No\'emie Cartier, Cesar Ceballos, and Vincent Pilaud.
\newblock Lattices of acyclic pipe dreams.
\newblock {\em Algebr. Comb.}, 8(3):817--856, 2025.

\bibitem[BPR12]{bergeron_higher_2012}
Fran\c{c}ois Bergeron and Louis-Fran\c{c}ois Pr\'eville-Ratelle.
\newblock Higher trivariate diagonal harmonics via generalized {T}amari posets.
\newblock {\em J. Comb.}, 3(3):317--341, 2012.

\bibitem[CD25]{doolittle}
Cesar Ceballos and Joseph Doolittle.
\newblock Subword complexes and {K}alai's conjecture on reconstruction of spheres.
\newblock {\em Discrete Comput. Geom.}, 74(1):23--48, 2025.

\bibitem[Ceb24]{Ceb24}
Cesar Ceballos.
\newblock A canonical realization of the alt $\nu$-associahedron.
\newblock Preprint available at \href{https://arxiv.org/abs/2401.17204v1}{arXiv:2401.17204v1}, 2024.

\bibitem[CFZ02]{chapoton_polytopal_2002}
Fr\'ed\'eric Chapoton, Sergey Fomin, and Andrei Zelevinsky.
\newblock Polytopal realizations of generalized associahedra.
\newblock {\em Canad. Math. Bull.}, 45(4):537--566, 2002.

\bibitem[CLS14]{ceballos_and_labbe}
Cesar Ceballos, Jean-Philippe Labb\'{e}, and Christian Stump.
\newblock Subword complexes, cluster complexes, and generalized multi-associahedra.
\newblock {\em J. Algebraic Combin.}, 39(1):17--51, 2014.

\bibitem[CP24]{ceballos_sweakorderII}
Cesar Ceballos and Viviane Pons.
\newblock The {$s$}-weak order and {$s$}-permutahedra {II}: {T}he combinatorial complex of pure intervals.
\newblock {\em Electron. J. Combin.}, 31(3):Paper No. 3.12, 59, 2024.

\bibitem[CPS19]{CPS19}
Cesar Ceballos, Arnau Padrol, and Camilo Sarmiento.
\newblock Geometry of {$\nu$}-{T}amari lattices in types {$A$} and~{$B$}.
\newblock {\em Trans. Amer. Math. Soc.}, 371(4):2575--2622, 2019.

\bibitem[CPS20]{ceballos_vTamari_subword_2020}
Cesar Ceballos, Arnau Padrol, and Camilo Sarmiento.
\newblock The {$\nu$}-{T}amari lattice via {$\nu$}-trees, {$\nu$}-bracket vectors, and subword complexes.
\newblock {\em Electron. J. Combin.}, 27(1):Paper No. 1.14, 31, 2020.

\bibitem[FZ03a]{clusteralgebrasII}
Sergey Fomin and Andrei Zelevinsky.
\newblock Cluster algebras. {II}. {F}inite type classification.
\newblock {\em Invent. Math.}, 154(1):63--121, 2003.

\bibitem[FZ03b]{fomin_ysystems_2003}
Sergey Fomin and Andrei Zelevinsky.
\newblock {$Y$}-systems and generalized associahedra.
\newblock {\em Ann. of Math. (2)}, 158(3):977--1018, 2003.

\bibitem[JS23]{bpolyhedra}
Dennis Jahn and Christian Stump.
\newblock Bruhat intervals, subword complexes and brick polyhedra for finite {C}oxeter groups.
\newblock {\em Math. Z.}, 304(2):Paper No. 24, 32, 2023.

\bibitem[KM04]{knutson_miller}
Allen Knutson and Ezra Miller.
\newblock Subword complexes in {C}oxeter groups.
\newblock {\em Adv. Math.}, 184(1):161--176, 2004.

\bibitem[KM05]{KnutsonMiller_Groebner_2005}
Allen Knutson and Ezra Miller.
\newblock Gr\"obner geometry of {S}chubert polynomials.
\newblock {\em Ann. of Math. (2)}, 161(3):1245--1318, 2005.

\bibitem[Mü25]{matthiasWebsite}
Matthias Müller.
\newblock Figures created using {SageMath3D}, personal website, 2025.

\bibitem[PP12]{pilaud_multitriangulations_2012}
Vincent Pilaud and Michel Pocchiola.
\newblock Multitriangulations, pseudotriangulations and primitive sorting networks.
\newblock {\em Discrete Comput. Geom.}, 48(1):142--191, 2012.

\bibitem[PRV17]{preville_vTamari_2017}
Louis-Fran\c Pr\'eville-Ratelle and Xavier Viennot.
\newblock The enumeration of generalized {T}amari intervals.
\newblock {\em Trans. Amer. Math. Soc.}, 369(7):5219--5239, 2017.

\bibitem[PS12]{bpofsn}
Vincent Pilaud and Francisco Santos.
\newblock The brick polytope of a sorting network.
\newblock {\em European J. Combin.}, 33(4):632--662, 2012.

\bibitem[PS15]{bpofssc}
Vincent Pilaud and Christian Stump.
\newblock Brick polytopes of spherical subword complexes and generalized associahedra.
\newblock {\em Adv. Math.}, 276:1--61, 2015.

\bibitem[Stu11]{stump_newperspective_2011}
Christian Stump.
\newblock A new perspective on {$k$}-triangulations.
\newblock {\em J. Combin. Theory Ser. A}, 118(6):1794--1800, 2011.

\bibitem[vBC24]{bell_framing_2024}
Matias von Bell and Cesar Ceballos.
\newblock Framing lattices and flow polytopes.
\newblock {\em S\'em. Lothar. Combin.}, 91B:Art. 98, 12, 2024.

\bibitem[vBCFdSG]{bell_framingtopes}
Matias von Bell, Cesar Ceballos, and Sergio~Alejandro Fernandez~de Soto~Guerrero.
\newblock Framingtopes.
\newblock Work in progress.

\end{thebibliography}


\end{document}